\documentclass[11pt, a4paper, reqno]{amsart}
\usepackage{amsmath}%
\usepackage{amsfonts}%
\usepackage{amssymb}%
\usepackage{graphicx}
\usepackage{amsthm}
\usepackage{amssymb}
\usepackage{tikz}
\usepackage{tikz-cd}
\usetikzlibrary{matrix,arrows}
\usepackage[english, ngerman]{babel}
\usepackage[utf8]{inputenc}
\usepackage{hyperref}
\usepackage{cite}
\usepackage{footnote}
\usepackage{stmaryrd}
\SetSymbolFont{stmry}{bold}{U}{stmry}{m}{n}
\usepackage[a4paper,margin=3cm]{geometry}
\usepackage{verbatim}
\usepackage{enumerate}

\newtheorem{thm}{Theorem}
\theoremstyle{plain}

\newtheorem{cor}[thm]{Corollary}

\newtheorem{lem}[thm]{Lemma}

\newtheorem{prop}[thm]{Proposition}

\theoremstyle{remark}
\newtheorem{rem}[thm]{Remark}
\newtheorem*{acknowledgements}{Acknowledgements}

\renewcommand{\mod}{\mathop{\rm mod}\nolimits}
\newcommand{\sign}{\mathop{\rm sign}\nolimits}

\renewcommand{\div}{\, | \,}
\newcommand{\ddiv}{\, || \,}

\newcommand{\notdiv}{\mathopen{\mathchoice
             {\not{|}\,}
             {\not{|}\,}
             {\!\not{\:|}}
             {\not{|}}
             }}

\newcommand{\SL}[1]{\mathop{\rm SL}_{#1} \nolimits}
\newcommand{\PSL}[1]{\mathop{\rm PSL}_{#1} \nolimits}

\renewcommand{\Re}{\mathop{\rm Re}\nolimits}
\renewcommand{\Im}{\mathop{\rm Im}\nolimits}

\newcommand{\quotient}[2]{
        \mathchoice
            {
                \text{\raise1ex\hbox{$#1$}\Big/\lower1ex\hbox{$#2$}}%
            }
            {
                #1\,/\,#2
            }
            {
                #1\,/\,#2
            }
            {
                #1\,/\,#2
            }
    }
		
\newcommand{\lquotient}[2]{
        \mathchoice
            {
                \text{\lower1ex\hbox{$#1$}\Big \backslash \raise01ex\hbox{$#2$}}%
            }
            {
                #1\,\backslash\,#2
            }
            {
                #1\,\backslash\,#2
            }
            {
                #1\,\backslash\,#2
            }
    }

\newcommand{\rquotient}[2]{
        \mathchoice
            {
                \text{\raise01ex\hbox{$#1$}\Big/\lower1ex\hbox{$#2$}}%
            }
						{
                #1\,/\,#2
            }
            {
                #1\,/\,#2
            }
            {
                #1\,/\,#2
            }
    }
		
\newcommand{\lrquotient}[3]{
        \mathchoice
            {
                \text{\lower1ex\hbox{$#1$}\Big \backslash \raise01ex\hbox{$#2$}\Big/\lower1ex\hbox{$#3$}}%
            }
            {
                #1\,\backslash\,#2\,/\,#3
            }
            {
                #1\,\backslash\,#2\,/\,#3
            }
            {
                #1\,\backslash\,#2\,/\,#3
            }
    }

\newcommand{\sm}{\left(\begin{smallmatrix}}
\newcommand{\esm}{\end{smallmatrix}\right)}
\newcommand{\bpm}{\begin{pmatrix}}
\newcommand{\ebpm}{\end{pmatrix}}

\newcommand{\starsum}{\sideset{}{^\star}\sum}

\numberwithin{equation}{section}
\begin{document}
\selectlanguage{english}

\bibliographystyle{alpha}

\title{Kloosterman sums do not correlate with periodic functions}
\author{Raphael S. Steiner}
\address{Department of Mathematics, ETH Z\"urich, 8092 Z\"urich, CH}%
\email{raphael.steiner.academic@gmail.com}%


\date{\today}


\begin{abstract} We provide uniform bounds for sums of Kloosterman sums in \emph{all} arithmetic progressions. As a consequence, we find that Kloosterman sums do not correlate with periodic functions.
\end{abstract}
\maketitle


\section{Introduction} \label{sec:intro}

The classical Kloosterman sum
$$
S(m,n;c) = \sum_{ad \equiv 1 \mod(c)} e\left( \frac{am+dn}{c} \right),
$$
where $m,n \in \mathbb{Z}$ are integers and $e(z)=\exp(2\pi i z)$, are a well-studied and widely applied source for cancellation. They themselves admit square-root cancellation due to Weil
\begin{equation}
	|S(m,n;c)| \le \tau(c) (m,n,c)^{\frac{1}{2}}c^{\frac{1}{2}},
	\label{eq:Ex-Weilbound}
\end{equation}
see for example \cite[Thm 9.2]{Knightly-Li}, and the normalised numbers $S(m,n;p)/p^{1/2}$, for $p$ a prime and $mn \neq 0$, are expected (and in some cases proven) to follow a Sato--Tate distribution, see for example \cite{MichelAutourKloos} and references therein. In a similar vein, Linnik \cite{LinnikKloosterman} and Selberg \cite{SelbergFourier} (independently) conjectured an additional square-root cancellation when summed up over the modulus when $mn\neq 0$\footnote{Here, we have stated the uniform version by Sarnak--Tsimerman \cite{SarTsim} with the $|mn|^{o(1)}$ safety factor.}:
\begin{equation}
\sum_{c \le C} \frac{S(m,n;c)}{c} \ll |mnC|^{o(1)}.
\label{eq:Linnik-Selberg-conjecture}
\end{equation}
Here and throughout, we make use of Vinogradov's notation. A slightly more extended version of this conjecture would imply sharp error terms in the prime geodesic theorem \cite{IwPrimeGeodesic} and an optimal covering exponent for $S^3$ \cite{S3covexp}, see also \cite{twistlinnik}.

Recently, it was conjectured by El Abdalaoui--Shparlinski--Steiner \cite{Klooster-low-complexity-seq} that Kloosterman sums should in general not correlate with low-complexity sequences. Whilst this was shown to be true for (vertical) averages over the entries of the Kloosterman sums, see \cite[Ex. 10.4]{Kowalski-notes-Klooster-low-complexity} or \cite[Thm 2.8]{Klooster-low-complexity-seq}, it is wide open for (horizontal) averages over the modulus in this generality. In this paper, we make progress towards this conjecture by completely resolving the case of periodic functions. We accomplish this is a strong sense by providing a power saving.

%
%

\begin{thm} Let $m,n \in \mathbb{Z}\backslash\{0\}$ be two non-zero integers and $F$ a periodic function. Then, we have
	$$
	\sum_{c \le C} \frac{S(m,n;c)}{\sqrt{c}} F(c) \ll_{m,n,F} C^{1-\frac{9}{32}+o(1)}.
	$$
	\label{thm:periodic-correlation}
\end{thm}

Natually, Theorem \ref{thm:periodic-correlation} is comprised of the building blocks where $F$ is taken to be the indicator function of an arithmetic progression. Here, we state only dependence on the arithmetic progression, however a completely uniform version is avaible at the end of the paper, see Theorem \ref{thm:mnqQ-depence}.

\begin{thm} Let $m,n \in \mathbb{Z}\backslash\{0\}$ be two non-zero integers. Further let $a,q,Q \in \mathbb{N}$ be integers with $q \mid Q$ and $(a,Q/q)=1$. Then, we have
	\begin{multline*}
		\sum_{\substack{c \le C \\ c \equiv aq \mod(Q)}} \frac{S(m,n;c)}{c} \\ \ll_{m,n} (QC)^{o(1)} \left( \frac{(q,Q/q)^{1/3}}{q^{1/3}}C^{1/6}+ \frac{Q^{1/2}(q,Q/q)^{1/2}}{q^{1/2}} \left( 1+ \frac{C}{Q^2}  \right)^{2\vartheta} \right),
	\end{multline*}
where $\vartheta$ is the best bound towards the Selberg eigenvalue conjecture on $\Gamma_0(Q)$.
	\label{thm:qQ-depence}
\end{thm}

Theorem \ref{thm:periodic-correlation} thus follows from an appeal to the Kim--Sarnak bound $\vartheta \le \frac{7}{64}$ \cite{KimSarnak}. The two extreme cases, $q=1$ and $q=Q$, in Theorem \ref{thm:qQ-depence} have been previously treated by Blomer--Mili{\'c}evi{\'c} \cite{BlomerKloos} and Deshouillers--Iwaniec \cite{IwDes}, Ganguly--Sengupta \cite{APKloosterman}, \cite{APKloostermanCorr}, respectively. In either case, the strength of Theorem \ref{thm:qQ-depence} is comparable to the one achieved by Blomer--Mili{\'c}evi{\'c} and Ganguly--Sengupta. The main novelty of Theorem \ref{thm:qQ-depence} is that it applies to \emph{all} arithmetic progressions, in particular also to `ramified' ones, such as $c \equiv p \mod(p^2)$. It is precisely these types of arithmetic progressions that are the most difficult and are at the heart of this paper.

The technical input, as in the previous works, is the Bruggeman--Kuznetsov trace formula. Deshouillers--Iwaniec and Ganguly--Sengupta were able to directly apply the formula for the group $\Gamma_0(Q)$, whereas Blomer--Mili{\'c}evi{\'c} proved and made use of an identity between the Kloosterman sum and its Dirichlet-twisted variants to which they were then able to apply the Bruggeman--Kuznetsov trace formula on $\Gamma_0(Q)$ with a Dirichlet character as multiplier. Later, Drappeau \cite{DrapKloost} and Kiral--Young \cite{Young-Kiral-Kloosterman} pointed out that Blomer--Mili{\'c}evi{\'c}'s Kloosterman identity may be replaced by a particular choice of cusps in the Bruggeman--Kuznetsov trace formula (again averaged over Dirichlet-character multpliers). We show that with an additional average over various cusps one is also able to resolve the `ramified' progressions. For this, we compute several Kloosterman sums associated to the group $\Gamma_0(Q)$ and Dirichlet character multiplier at various cusps in more generality than the ones considered by Drappeau and Kiral--Young. In order to get stronger error bounds, we also compute and bound Kloosterman sums associated to the group
$$
\Gamma_{0,\pm 1}(Q;Q/q) = \left\{ \gamma \in \SL{2}(\mathbb{Z}) | \gamma \equiv \sm \ast & \ast \\ 0 & \ast \esm \mod(Q),\ \gamma \equiv \pm \sm 1 & \ast \\ 0 & 1 \esm \mod(\tfrac{Q}{q})\right\}
$$
with trivial multiplier system at various cusps.

Finally, a plausible next class of low-complexity sequences to consider are those stemming from a finite automaton. Here, it would suffice prove cancellation in the two-autocorrelation, see \cite{DrapAutomaton}. However, the lack of an adequate substitution for the Bruggeman--Kuznetsov formula makes this particularly challenging.



\subsection{Overview} In \S\ref{sec:auto-kloos} and \S\ref{sec:kloosbound}, we compute and bound the relevant Kloosterman sums. Subsequently, we make use of these bounds to estimate the Fourier coefficients of automorphic forms in \S\ref{sec:maassmain} and \S\ref{sec:holmain}. Finally, in \S\ref{sec:final}, we apply the Bruggeman--Kuznetsov trace formula and establish Theorem \ref{thm:qQ-depence}.

\begin{acknowledgements} I would like to thank Igor Shparlinski and El Houcein El Abdalaoui for fruitful discussions and encouragement, and Sary Drappeau for comments on an earlier draft. Furthermore, I would like to extend my gratitude to my employer, the Institute
	for Mathematical Research (FIM) at ETH Z\"urich.
\end{acknowledgements}

\section{Kloosterman sums} \label{sec:auto-kloos}

Let $-I \in \Gamma \subseteq \SL{2}(\mathbb{R}) $ be a Fuchsian group of the first kind, $\kappa \in \mathbb{R}$, and $\upsilon$ a multiplier system of weight $\kappa \in \{0,1\}$, i.e.\@ a character $\upsilon: \Gamma \to \{z \in \mathbb{C} \, | \, |z|=1\}$ with $\upsilon(-I)= (-1)^{\kappa}$. Let $B=\{ \sm 1 & n \\ & 1 \esm \in \SL{2}(\mathbb{R}) \ | \ n \in \mathbb{Z}\}$. Fix for each cusp $\mathfrak{a}$ of $\Gamma$, a scaling matrix $\sigma_{\mathfrak{a}}$. That is $\sigma_{\mathfrak{a}} \in \SL{2}(\mathbb{R})$ satisfying
\begin{itemize}
	\item $\sigma_{\mathfrak{a}} \infty = \mathfrak{a}$, 
	\item $\gamma_{\mathfrak{a}}:=\sigma_{\mathfrak{a}} \sm 1 & 1 \\ & 1 \esm \sigma_{\mathfrak{a}}^{-1}$ projects to a generator of the stabiliser $\widehat{\Gamma}_{\mathfrak{a}}$ of $\mathfrak{a}$ in $\widehat{\Gamma}$, where ${\ }\widehat{{}}{\ }$ denotes the projection $\SL{2}(\mathbb{R}) \to \PSL{2}(\mathbb{R})$,
\end{itemize}
We also assume a consistency relationship:
\begin{itemize}
	\item $\sigma_{\gamma \mathfrak{a}} = \gamma \sigma_{\mathfrak{a}}$ for all cusps $\mathfrak{a}$ and $\gamma \in \Gamma$.
\end{itemize}
We shall denote by $\eta_{\mathfrak{a}}\in [0,1[$ the cusp parameter of the cusp $\mathfrak{a}$, which is characterised by $e(\eta_{\mathfrak{a}})=\upsilon(\gamma_{\mathfrak{a}})$. A cusp $\mathfrak{a}$ is called \emph{singular} if and only if $\eta_{\mathfrak{a}}=0$. For two cusps $\mathfrak{a}, \mathfrak{b}$, we denote the set
$$
\mathcal{C}_{\mathfrak{a},\mathfrak{b}} = \left\{ c \in \mathbb{R}^+ | \exists \sm \ast & \ast \\ c & \ast \esm \in \sigma_{\mathfrak{a}}^{-1} \Gamma \sigma_{\mathfrak{b}} \right\}.
$$
Given two additional integers $m,n$, and $c\in \mathcal{C}_{\mathfrak{a},\mathfrak{b}}$, the Kloosterman sums are defined by (cf. \cite[\S 3.3]{thesis} or \cite{GoldfeldSarnak} with slightly diffent normalisation).
\begin{equation}
	S^{\Gamma,\upsilon,\kappa}_{\mathfrak{a},\mathfrak{b}}(m,n;c) = e^{-\frac{\pi i}{2}\kappa} \sum_{\sm a & b \\ c & d \esm \in B \backslash \sigma_{\mathfrak{a}}^{-1} \Gamma \sigma_{\mathfrak{b}} \slash B } \overline{\upsilon(\sigma_{\mathfrak{a}} \sm a & b \\ c & d \esm \sigma_{\mathfrak{b}}^{-1})} e\left( (m+\eta_{\mathfrak{a}})\frac{a}{c}+(n+\eta_{\mathfrak{b}})\frac{d}{c} \right). 
	\label{eq:genKloos}
\end{equation}
If the group $\Gamma$ is clear from the context, we shall drop it from the superscript and simply write $S^{\upsilon,\kappa}_{\mathfrak{a},\mathfrak{b}}(m,n;c)$.

\begin{rem}
	The cusp parameters are independent of the choice of scaling matrix and further only depend on their $\Gamma$-orbit of the cusp, see \cite[\S 3.3]{MFaF}. However, the Kloosterman sums and the Fourier coefficients later on are not. In particular, we have
	$$
	S^{\upsilon,\kappa}_{\sigma_{\mathfrak{a}}u(\alpha),\sigma_{\mathfrak{b}}u(\beta)}(m,n;c) = e\left( -(m+\eta_{\mathfrak{a}})\alpha+(n+\eta_{\mathfrak{b}}) \beta \right) S^{\upsilon,\kappa}_{\sigma_{\mathfrak{a}},\sigma_{\mathfrak{b}}}(m,n;c),
	$$
	where $u(\alpha)= \sm 1& \alpha \\ & 1 \esm$.
\end{rem}

\begin{rem}
	The Kloosterman sums satisfy $\overline{S^{\upsilon,\kappa}_{\mathfrak{a},\mathfrak{b}}(m,n;c)}=S^{\upsilon,\kappa}_{\mathfrak{b},\mathfrak{a}}(n,m;c)$, see \cite[Prop. 3.3.2]{thesis}. Note also $\mathcal{C}_{\mathfrak{a},\mathfrak{b}}=\mathcal{C}_{\mathfrak{b},\mathfrak{a}}$, cf.\@ \cite[Lem 3.2.7]{thesis}. In particular, the Kloosterman sums $S^{\upsilon,\kappa}_{\mathfrak{a},\mathfrak{a}}(m,m;c)$ are real.
	\label{rem:Kloos-conj}
\end{rem}

\subsection{Evaluations for $\Gamma=\Gamma_0(Q)$, $\upsilon=\chi$, and $\kappa \in \{0,1\}$ with $\chi(-1)=(-1)^{\kappa}$} \label{sec:Klos-Gamma_0Q} We let
$$
\Gamma = \Gamma_0(Q) = \left\{\gamma \in \SL{2}(\mathbb{Z}) \, | \, \gamma \equiv \sm \ast & \ast \\ 0 & \ast \esm \, \mod(Q) \right\}
$$
be a congruence subgroup. For a Dirichlet character $\chi$ of modulus $Q$, we define the multiplier system $\upsilon=\upsilon_{\chi}$ by
$$
\upsilon_{\chi}\left( \sm a & b \\ c & d \esm \right) = \chi(d), \quad \forall \sm a & b \\ c & d \esm \in \Gamma_0(Q).
$$
We shall often simply write $\chi$ instead of $\upsilon_{\chi}$. We let the weight $\kappa= 0,1$ depending on whether $\chi$ is even, respectively odd. A representative set of the set of cusps of $\Gamma_0(Q)$ is given by $r/q$ with $q|Q$, $(r,q)=1$, and $r \mod((q,Q/q))$ (see for example \cite[Prop. 2.6]{IwClassic}). We remark that the cusp $\infty$ is equivalent modulo $\Gamma$ to $1/Q$ and that the representatives $r/q$ may be chosen in a way to further satisfy $(r,Q/q)=1$. The cusp $r/q$ has width $w_{q}=\frac{Q}{(Q,q^2)}$ and a generator for the stabiliser of the cusp is given by
\begin{equation}
\gamma_{r/q}= \begin{pmatrix} 1- rqw_{q} & r^2w_{q} \\ - q^2w_{q}   & 1+ rqw_{q} \end{pmatrix}.
\label{eq:gamma-r/q}
\end{equation}
From now on, we assume that the Dirichlet character $\chi$ is induced by one of modulus $Q/q$, such that $\chi(1+rqw_q)=\chi(1)=1$ as $qw_q=\frac{Q}{q} \cdot \frac{q^2}{(Q,q^2)}$. Hence, the cusps $r/q$ are singular. Likewise, the cusp $\infty$ is singular. Scaling matrices for the cusps $r/q$ may be given by
\begin{equation}
\sigma_{r/q} = \begin{pmatrix} r & x \\ q & y \end{pmatrix} \begin{pmatrix} \sqrt{w_{q}} & \\ & 1/\sqrt{w_{q}}  \end{pmatrix},
\label{eq:sigma-r/q}
\end{equation}
where $x,y$ are integers with $ry-qx=1$. We may and shall assume that $\frac{Q}{q} \div x$, so that $ry \equiv 1 \mod(Q)$. 
 For the cusp $\infty$, we choose the scaling matrix $\sigma_{\infty}=I$.
 \begin{lem} With the above choices, we have
 	$$
 	\mathcal{C}_{\infty,r/q} = \{ cq \sqrt{w_{q}}:c \in \mathbb{N}, (c,\tfrac{Q}{q})=1 \}.
 	$$
 	and for $cq\sqrt{w_q} \in \mathcal{C}_{\infty,r/q}$, we have
 	$$
 	S_{\infty,r/q}^{\chi,\kappa}(m,nw_{q};cq\sqrt{w_{q}}) = (-i)^{\kappa}  \overline{\chi(c)} \sum_{\substack{ad \equiv 1 \mod(cq) \\  d \equiv c\overline{r} \mod((q,Q/q))}} e\left( \frac{ma+nd}{cq} \right),
 	$$
 	$$
 		\sum_{\substack{r \mod((q,Q/q)) \\ (r,Q)=1}} S^{\chi,\kappa}_{\infty,r/q}(m,nw_q;cq\sqrt{w_q}) = (-i)^{\kappa} \overline{\chi(c)} S(m,n;cq).
 	$$
 	\label{lem:Gamma0Q-Kloosterman}
 \end{lem}
 \begin{proof} A straightforward calculation yields
 	$$
 	\sigma_{\infty}^{-1}\Gamma \sigma_{r/q} = \left\{ \begin{pmatrix} a\sqrt{w_{q}} & b/\sqrt{w_{q}} \\ cq\sqrt{w_{q}} & d/\sqrt{w_{q}}  \end{pmatrix} \in \SL{2}(\mathbb{R}) : a,b,c,d\in \mathbb{Z},  \tfrac{Q}{q} \div cy-d  \right\}.
 	$$
 	These matrices must satisfy $ad-cqb=1$, hence $(c,d)=1$. Furthermore, $c \equiv rd \, \mod(\frac{Q}{q})$, thus $(c,\frac{Q}{q})=1$. In reverse, given $c$ relatively prime to $\frac{Q}{q}$, we may find $d \in \mathbb{Z}$ with $d \equiv cy \, \mod(\frac{Q}{q})$ and $(d,cq)=1$. Subsequently, by B\'ezout, we may find $a,b \in \mathbb{Z}$ such that $ad-cqb=1$. Hence, $\mathcal{C}_{\infty,r/q}$ is of the claimed shape. Inserting into the definition for the Kloosterman sum \eqref{eq:genKloos}, we find
 $$
 		S_{\infty,r/q}^{\chi,\kappa}(m,nw_{q};cq\sqrt{w_{q}}) = (-i)^{\kappa}   \sum_{\substack{a \mod(cq) \\ d \mod(cqw_q) \\ ad \equiv 1 \mod(cq)\\ d \equiv cy \mod(Q/q)}} \overline{\chi(dr-cqx)} e\left( \frac{ma+nd}{cq} \right)
 		$$%
We now have $dr-cqx \equiv c(yr-qx) \equiv c \mod(Q/q)$ and $y \equiv \overline{r} \, \mod(\frac{Q}{q})$. Furthermore, the exponential only depends on $d \, \mod(cq)$ and we claim that there is a unique lift of $d \, \mod(cq)$ satisfying $d\equiv c\overline{r} \, \mod((q,Q/q))$ to $d \, \mod(cqw_q)$ satisfying $d \equiv c\overline{r} \mod(Q/q)$. The latter is true since
 	$$
 	d+\ell cq \equiv c\overline{r} \mod(Q/q) \Leftrightarrow \frac{d-c\overline{r}}{(q,Q/q)} \equiv -\ell c \frac{q}{(q,Q/q)} \mod(w_q)
 	$$
	uniquely determines $\ell \mod(w_q)$ as both $c$ and $q/(q,Q/q)$ are invertible modulo $w_q$. This proves the claimed expression of the Kloosterman sum. The last equality of the lemma is clear.
	\end{proof}

\subsection{{Evaluations for $\Gamma=\Gamma_{0,\pm1}(Q;Q/q)$, $\upsilon=\sign_{Q/q}^{\kappa}$, and $\kappa \in \{0,1\}$}} \label{sec:Klooster-Gamma01-Q-Q/q}

Let $q \div Q$ and if $\frac{Q}{q}=1,2$ assume furthermore $\kappa=0$. Let
$$
\Gamma = \Gamma_{0,\pm 1}(Q;Q/q) = \left\{ \gamma \in \SL{2}(\mathbb{Z}) | \gamma \equiv \sm \ast & \ast \\ 0 & \ast \esm \mod(Q),\ \gamma \equiv \pm \sm 1 & \ast \\ 0 & 1 \esm \mod(\tfrac{Q}{q})\right\}.
$$
For $\gamma=\sm a & b \\ c & d \esm  \in \Gamma$, we set
$$
\sign_{Q/q}^{\kappa}(\gamma)=\sign_{Q/q}^{\kappa}(d) = \begin{cases} 1, & d \equiv 1 \mod(\tfrac{Q}{q}), \\ (-1)^{\kappa}, & d \equiv -1 \mod(\tfrac{Q}{q}). \end{cases}
$$
As in \S\ref{sec:Klos-Gamma_0Q}, we shall be considering the cusps $\infty$ and $r/q$ with $(r,Q)=1$, $r \mod(q,Q/q)$. Since the generator $\gamma_{r/q}$ of the stabiliser of the cusp $r/q$ in $\Gamma_0(Q)$ is already contained in $\Gamma_{0,\pm1}(Q,Q/q)$ it follows, that the width of the cusp $r/q$ is the same with respect to either group, i.e.\@ equal to $w_q$. The same goes for the cusp $\infty$. Thus, we may choose the same scaling matrices, i.e.\@ $\sigma_{\infty}=I$ and $\sigma_{r/q}$ given by \eqref{eq:sigma-r/q}.

\begin{lem} With the above choices, we have
	$\mathcal{C}_{\infty,\infty}=Q\mathbb{N}$ and for $c \in \mathcal{C}_{\infty,\infty}$:
	\begin{equation*}
		S^{\sign_{Q/q}^{\kappa},\kappa}_{\infty,\infty}(m,n;c) = (-i)^{\kappa} \sum_{\substack{ad \equiv 1 \mod(c)\\ a \equiv d \equiv \pm 1 \mod(\frac{Q}{q}) }} \sign_{Q/q}^{\kappa}(d) e\left( \frac{ma+nd}{c} \right).
	\end{equation*}
We also have $\mathcal{C}_{r/q,r/q}=Q\mathbb{N}$ and for $cqw_q \in Q \mathbb{N}$
\begin{multline}
S_{r/q,r/q}^{\sign_{Q/q}^{\kappa},\kappa}(mw_q,nw_q;cqw_q) = (-i)^{\kappa} (c,w_q) \\ \times  \sum_{\substack{a,d \mod (cq) \\ a+cy \equiv d-cy \equiv \pm 1  \mod((cq,Q/q)) \\ ad \equiv 1 \mod(cq)}} 
 \sign_{(cq,Q/q)}^{\kappa}(a+cy) e\left( \frac{ma+nd}{cq} \right) 
\label{eq:Kloos-Gamma01-Q-Q/q-indv}
\end{multline}
and
\begin{multline*}
\sum_{\substack{r \mod((q,Q/q))\\ (r,Q)=1}}S_{r/q,r/q}^{\sign_{Q/q}^{\kappa},\kappa}(mw_q,nw_q;cqw_q) = (-i)^{\kappa} (c,w_q) \frac{\varphi((q,Q/q))}{\varphi\left(\left(q, \frac{Q}{q(c,Q/q)} \right)\right)}  \\
\times \sum_{(c,\frac{Q}{q}) \div f \div (cq,\frac{Q}{q})} \mu\left(\frac{f}{(c,Q/q)}\right) \sum_{\substack{a,d \mod (cq) \\ a\equiv d \equiv \pm 1  \mod(f) \\ ad \equiv 1 \mod(cq)}} 
\sign_{f}^{\kappa}(a) e\left( \frac{ma+nd}{cq} \right)
\end{multline*}

	\label{lem:Kloos-Gamma01-Q-Q/q}
\end{lem}
\begin{proof}
	The claim for the cusp pair $(\infty, \infty)$ follows straight from the definitions. Let us now consider the cusp pair $(r/q,r/q)$. We have that
	\begin{multline*}
	\sigma_{r/q}^{-1} \begin{pmatrix} A & B \\ CQ & D \end{pmatrix} \sigma_{r/q} \\= \begin{pmatrix} Ary+Bqy-CQrx-Dqx & (Axy+By^2-CQx^2-Dxy)/w_q \\ (-Aqr-Bq^2+CQr^2+Dqr)w_q & -Aqx-Bqy+CQrx+Dry \end{pmatrix},
	\end{multline*}
	which is of the shape $\sm a & b/w_q \\ cqw_q & d \esm$ with $a,b,c,d \in \mathbb{Z}$ with $ad-cqb=1$. We claim furthermore that the conditions
	\begin{equation}
	\tfrac{Q}{q} \div (a-d)y-bq+cy^2
	\label{eq:lower-left-div}
\end{equation}
and	
\begin{equation}
	a+cy \equiv d-cy \equiv \pm 1 \mod(\tfrac{Q}{q})
	\label{eq:a-d-congruence}
\end{equation}	
	are necessary and sufficient. A short calculation reveals that they are indeed necessary. In order to show that they are sufficient, we compute
	$$
	\sigma_{r/q}\begin{pmatrix} a & b/w_q \\ cqw_q & d \end{pmatrix} \sigma_{r/q}^{-1} = \begin{pmatrix} ary-bqr+cqxy-dqx & -arx+br^2-cqx^2+drx \\ aqy-bq^2+cqy^2-dqy & -aqx+bqr-cqxy+dry \end{pmatrix}.
	$$
	The condition \eqref{eq:lower-left-div} shows that the lower left entry is divisible by $Q$. The conditions \eqref{eq:a-d-congruence} imply that the top left, respectively, bottom right entry satisfy
	\begin{multline*}
		ary-bqr+cqxy-dqx \\ \equiv ary-(ay+cy^2-dy)r+cqxy-dqx   
		\equiv d-cy  \equiv \pm 1, \ \mod(\tfrac{Q}{q})
	\end{multline*}
and
	\begin{multline*}
 -aqx+bqr-cqxy+dry  \\ \equiv  -aqx+(ay+cy^2-dy)r-cqxy+dry \equiv a+cy \equiv \pm 1, \ \mod(\tfrac{Q}{q}),
	\end{multline*}
where the choice of sign is the same. We note that \eqref{eq:lower-left-div} and \eqref{eq:a-d-congruence} together imply $(q,Q/q) \div cy^2$ and since $(y,Q)=1$ even $(q,Q/q) \div c$. This shows $\mathcal{C}_{r/q,r/q} \subseteq Q\mathbb{N}$. On the other hand for $cqw_q \in Q\mathbb{N}$, we may find $a \in \mathbb{Z}$ satisfying $a \equiv \pm 1-cy \ \mod(\frac{Q}{q})$ and $(a,cq)=1$, since for any prime $p \div (c,Q/q)$ we have $a \equiv \pm 1 \ \mod(p)$ and for any prime $p \div (q,Q/q)$ we have $a \equiv \pm1 \mp (1-ry) \equiv \pm ry \equiv \pm 1 \ \mod(p)$. We may now find integers $b,d$ such that $ad-cbq = 1$. Any other solution of the latter is of the shape $b+\ell a, d+\ell cq$ with $\ell \in \mathbb{Z}$. We have that $d \equiv \overline{a} \equiv a \mod((q,Q/q))$, since \eqref{eq:a-d-congruence} implies $a \equiv \pm 1 \mod((q,Q/q))$. Therefore,
$$
	ay-bq+cy^2-dy \equiv 0  \mod((q,Q/q)).
$$
Thus,
\begin{equation*}
	\tfrac{Q}{q} \div ay-(b+\ell a)q+cy^2-(d+\ell cq)y
	\Leftrightarrow \tfrac{Q}{q(q,Q/q)} \div \tfrac{ay-bq+cy^2-dy}{(q,Q/q)}-\ell \tfrac{q}{(q,Q/q)} (a+cy)
\end{equation*}
is solvable for $\ell$. Hence, $\mathcal{C}_{r/q,r/q}=Q \mathbb{N}$.

For $cqw_q \in \mathcal{C}_{r/q,r/q}$, we may explicate the definition of the Kloosterman sum to find
$$
S^{\sign_{Q/q}^{\kappa},\kappa}_{r/q,r/q}(mw_q,nw_q;cqw_q) = (-i)^{\kappa} \sum_{\substack{a,d \mod(cqw_q)\\ a+cy \equiv d-cy \equiv \pm 1 \mod(Q/q) \\ ad \equiv 1 \mod(cq) \\ \frac{Q}{q} \div cy^2 +ay-dy+\frac{1-ad}{c}}} \sign_{Q/q}^{\kappa}(a+cy) e\left( \frac{ma+nd}{cq} \right).
$$
The last condition in the sum may be rewritten as $(a+cy)(d-cy) \equiv 1 \ \mod(c \tfrac{Q}{q})$. Next, we claim that for $\epsilon \in \{\pm 1\}$ and $cqw_q \in \mathcal{C}_{r/q,r/q}$ the map
\begin{equation}
\begin{Bmatrix} a,d \ \mod (cqw_q) \\ a+cy \equiv d-cy \equiv \epsilon \ \mod(Q/q) \\
ad \equiv 1 \ \mod(cq) \\ (a+cy)(d-cy) \equiv 1 \ \mod(c\tfrac{Q}{q})
\end{Bmatrix} \mapsto \begin{Bmatrix} a,d \ \mod (cq) \\ a+cy \equiv d-cy \equiv \epsilon \ \mod((cq,Q/q)) \\
ad \equiv 1 \ \mod(cq) \\ (a+cy)(d-cy) \equiv 1 \ \mod(c(q,Q/q))
\end{Bmatrix}
\label{eq:congruence-map}
\end{equation}
is surjective and $(c,w_q)$ to $1$. Let $a,d$ be two representatives modulo $cq$. We shall write $a+\ell_1cq$ and $d+\ell_2cq$ and count how many $\ell_1,\ell_2 \ \mod(w_q)$ satisfy the congruences on the left-hand side of \eqref{eq:congruence-map}. The congruence $ad \equiv 1 \ \mod(cq)$ is clearly satisfied.
$$
a+cy+\ell_1cq \equiv \epsilon \ \mod(Q/q) \Leftrightarrow -\frac{a+cy-\epsilon}{(cq,Q/q)} \equiv \ell_1 \frac{cq}{(cq,Q/q)} \ \mod\left( \frac{w_q}{(c,w_q)} \right)
$$
shows that $\ell_1$ is uniquely defined modulo $w_q/(c,w_q)$. Thus, upon redefining $a$ we may assume that $a+cy \equiv \epsilon \ \mod(Q/q)$ and consider $a+\ell_1cq\frac{w_q}{(c,w_q)}$ with $\ell_1$ modulo $(c,w_q)$. We may argue similarly for $\ell_2$. Since $[c(q,Q/q),Q/q]=c\tfrac{Q}{q}/(c,w_q)$, we have
\begin{multline*}
	\left(a+cy+\ell_1cq\tfrac{w_q}{(c,w_q)}\right)\left(d-cy+\ell_2cq\tfrac{w_q}{(c,w_q)}\right) \equiv 1 \ \mod(c\tfrac{Q}{q}) \\
	\Leftrightarrow \frac{(a+cy)(d-cy)-1}{c\tfrac{Q}{q}/(c,w_q)}+\left( (d-cy) \ell_1 +(a+cy)\ell_2 \right) \frac{q^2}{(q^2,Q)}+ \ell_1\ell_2 cq \frac{w_q}{(c,w_q)}\frac{q^2}{(q^2,Q)}  \equiv 0 \ \mod((c,w_q)) \\
	\Leftrightarrow \frac{(a+cy)(d-cy)-1}{c\tfrac{Q}{q}/(c,w_q)}+\epsilon (\ell_1+\ell_2) \frac{q^2}{(q^2,Q)}  \equiv 0 \ \mod((c,w_q)).
\end{multline*}
Now, $q^2/(q^2,Q)=q/(q,Q/q)$ and $w_q=Q/(q(q,Q/q))$ are relatively prime, from which we deduce the claim. Finally, using $ad \equiv 1 \ \mod(cq)$ and $(q,Q/q) \div c$, we have
$$
(a+cy)(d-cy) \equiv 1 \ \mod(c(q,Q/q)) \Leftrightarrow a \equiv d \ \mod((q,Q/q)),
$$
which is implied by the condition $a+cy \equiv d-cy \equiv \pm 1 \mod((cq,Q/q))$. Hence, the expression of the Kloosterman sum in the lemma.

To prove the last equality, we first note that the Kloosterman sum \eqref{eq:Kloos-Gamma01-Q-Q/q-indv} only depends on $y$ (respectively $r$) modulo $(q,\tfrac{Q}{q(c,Q/q)}) \div (q,Q/q)$. We further claim that for $\epsilon \in \{\pm 1\}$ and $(q,Q/q) \div c$ the map
\begin{equation}
	\begin{Bmatrix} a,d \ \mod (cq) \\ ad \equiv 1 \ \mod(cq) \\ \exists y \in (\mathbb{Z} / (q,\tfrac{Q}{q(c,Q/q)})\mathbb{Z})^{\times}:\\
		 a+cy \equiv d-cy \equiv \epsilon \ \mod((cq,Q/q))
	\end{Bmatrix}
	\mapsto
	\begin{Bmatrix} a,d \ \mod (cq) \\ ad \equiv 1 \ \mod(cq) \\ a \equiv d \equiv \epsilon \ \mod((c,Q/q)) \\
		\left(\frac{a-\epsilon}{(c,Q/q)} , (q, \frac{Q}{q(c,Q/q)}) \right) =1
	\end{Bmatrix}
	\label{eq:congruence-map-2}
\end{equation}
is one-to-one. This map is well-defined since $$\frac{a-\epsilon}{(c,Q/q)} \equiv -\frac{c}{(c,Q/q)} y \ \mod\left( (q, \tfrac{Q}{q(c,Q/q)}) \right)$$
and the right-hand side is invertible. We shall now show that the map is bijective. Let $a, d$ be in the co-domain. We may write $a+cy \equiv d-cx \equiv \epsilon \ \mod((cq,Q/q))$ with $x,y \ \mod((q,\tfrac{Q}{q(c,Q/q)}))$ and $(y,(q,\tfrac{Q}{q(c,Q/q)}))=1$. We find
$$\begin{aligned}
1 \equiv \epsilon^2 \equiv (a+cy)(d-cx) &\equiv ad +c(dy-ax)-c^2xy &\ \mod((cq,Q/q)) \\
&\equiv 1 +c(dy-ax) &\ \mod((cq,Q/q))
\end{aligned}$$
as $(cq,Q/q) \div c (q,Q/q) \div c^2$. Hence, $x \equiv \epsilon dx \equiv \epsilon dy \equiv y \ \mod((q,\tfrac{Q}{q(c,Q/q)}))$ since we have $(q,\tfrac{Q}{q(c,Q/q)}) \div (q,Q/q) \div (c,Q/q)$. In order to conclude the lemma, we resolve the condition $\left(\frac{a-\epsilon}{(c,Q/q)} , (q, \frac{Q}{q(c,Q/q)}) \right) =1$ via M\"obius inversion.

\end{proof}

\section{Bounds for Kloosterman sums} \label{sec:kloosbound}
 The classical Kloosterman sums satisfy the Weil bound
\begin{equation}
	|S(m,n;c)| \le \tau(c) (m,n,c)^{\frac{1}{2}}c^{\frac{1}{2}} \tag{\ref{eq:Ex-Weilbound}},
\end{equation}
see \cite[Thm 9.2]{Knightly-Li}. The goal of this section is to establish similar bounds for the Kloosterman sums arising in \S\ref{sec:Klooster-Gamma01-Q-Q/q}. We require a bound for Gau{\ss} sums.

\begin{lem} Let $a,b,c \in \mathbb{Z}$ with $c>0$. Then, the Gau{\ss} sum
	$$
	\mathcal{G}(a,b;c) = \sum_{n \mod(c)} e \left( \frac{an^2+bn}{c} \right)
	$$
	is $0$ unless $(a,c) \div b$, in which case one has
	$$
	\mathcal{G}(a,b;c) = (a,c) \mathcal{G}(a/(a,c),b/(a,c);c/(a,c)).
	$$
	Furthermore, we have the general bound
	$$
	|\mathcal{G}(a,b;c)| \le (a,c)^{1/2}c^{1/2} \begin{cases} 1, & 2 \notdiv c/(a,c), \\ 2^{1/2}, & 2 \div c/(a,c). \end{cases}
	$$
	\label{lem:Gauss}
\end{lem}
\begin{proof} Let $d=(a,c)$ and write $a=a'd$, $c=c'd$. We further decompose $n= m+\ell c'$ with $m \, \mod(c')$ and $\ell \, \mod(d)$. Thus,
	$$
	\mathcal{G}(a,b;c) = \sum_{m \mod(c')} e\left(\frac{a'm^2}{c'}+\frac{bm}{c}\right) \sum_{\ell \mod(d)} e\left(\frac{b\ell}{d}\right) = \begin{cases} 0, & d \notdiv b, \\ d \mathcal{G}(a',b/d;c'), & d \div b. \end{cases}
	$$
	For $(a,c)=1$, the Gau{\ss} sums may be evaluated explicitly (cf. \cite[Lem 3]{BaiBrow}), from which one may deduce the bound
	$$
	|\mathcal{G}(a,b;c)| \le \begin{cases} \sqrt{c}, & 2 \notdiv c, \\ \sqrt{2c}, & 2 \div c. \end{cases}
	$$
\end{proof}

\begin{lem} Let $m,n,c \in \mathbb{Z}$ be three integers with $c>0$. Further, let $q \div c$ be a divisor, $f \in \mathbb{Z}$ with $(f,q)=1$, and $\kappa \in \{0,1\}$ such that $\kappa=0$ if $q=1$ or $2$. Then, the Kloosterman sum
	$$
	T_{f}(m,n;q \div c) = \starsum_{\substack{a \mod(c) \\ a \equiv f \mod(q)}} e \left( \frac{ma+n\overline{a}}{c} \right)
	$$
	satisfies the bound
	$$
	|T_f(m,n;q \div c)| \le 2^{3/2} \tau(c) \min\left\{\tfrac{c}{q}, (\tfrac{c}{q},m,n)^{1/2} c^{1/2}\right\}.
	$$
	Here, and throughout the $\star$ in the summation indicates that the sum is over residues coprime to the modulus, i.e.\@ $(a,c)=1$ in this case.
	\label{lem:Klooster-T-bound}
\end{lem}

\begin{rem} The constant $2^{3/2}$ can likely be improved by some more careful analysis at everybody's favourite even prime.
\end{rem}

\begin{cor} We have
	$$
	|S^{\Gamma_0(Q),\sign_{Q/q}^{\kappa},\kappa}_{\infty,\infty}(m,n;c)| \le 2^{5/2} \tau(c) \min\left\{\tfrac{c}{Q/q} ,  (\tfrac{c}{Q/q},m,n)^{1/2} c^{1/2}\right\} .
	$$
	and
	\begin{multline*}
	\left |\sum_{\substack{r \mod((q,Q/q))\\ (r,Q)=1}}S_{r/q,r/q}^{\Gamma_{0,\pm1}(Q;Q/q),\sign_{Q/q}^{\kappa},\kappa}(mw_q,nw_q;cqw_q) \right | \\
	 \ll (cQ)^{o(1)} (c,Q/q) \min\left\{ \tfrac{cq}{(c,Q/q)} , \left(\tfrac{cq}{(c,Q/q)},m,n\right)^{1/2} (cq)^{1/2} \right\}.
	\end{multline*}
	\label{cor:bound-kloosterman-after-CS}
\end{cor}

\begin{proof}[Proof of Lemma \ref{lem:Klooster-T-bound}] Let us first note that for $q=1,2$ we have that
	$$
	T_f(m,n; q \div c) = S(m,n;c),
	$$
	the classical Kloosterman sum, to which we may apply the Weil bound \eqref{eq:Ex-Weilbound}. Furthermore, we have the following twisted multiplicativity relation: for $c=rs$ composite with $(r,s)=1$, we have
	$$
	T_f(m,n;q \div c) = T_f(m\overline{s},n\overline{s};(q,r)\div r) T_f(m\overline{r},n\overline{r};(q,s) \div s),
	$$
	where $\overline{r}r+\overline{s}s=1$. Thus, it suffices to look at the case $c=p^{\alpha}$, a power of a prime, and $q=p^{\gamma}$ with $\gamma \ge 1$. Let $p^{\delta}=(m,n,p^{\alpha})$, then
	$$
	T_f(m,n; p^{\gamma} \div p^{\alpha}) = p^{\min\{\alpha-\gamma,\delta\}} T_f(m/p^{\delta},n/p^{\delta}; p^{\min\{\alpha-\delta,\gamma\}} \div p^{\alpha-\delta}).
	$$
	Thus, we may additionally assume that $(m,n,p)=1$. For $\alpha\le 2 \gamma$, we have
	\begin{multline*}
	T_f(m,n;p^{\gamma} \div p^{\alpha}) = e\left(\frac{mf+n\overline{f}}{p^{\alpha}}\right) \sum_{\ell \mod(p^{\alpha-\gamma})} e\left(\frac{(m-n\overline{f}^2)\ell}{p^{\alpha-\gamma}}\right) \\
	 = \begin{cases} 0, & p^{\alpha-\gamma} \notdiv f^2m-n, \\ p^{\alpha-\gamma} e\left(\frac{mf+n\overline{f}}{p^{\alpha}}\right), & p^{\alpha-\gamma} \div f^2m-n. \end{cases}
	\end{multline*}
	We may assume now that $\alpha \ge 2\gamma+1$. We shall write $\alpha=2 \beta+\gamma+\epsilon$ with $\epsilon \in \{0,1\}$. We remark, that the assumption $\alpha \ge 2 \gamma$ implies $\beta \ge \gamma/2$ and thus $\beta \ge 1$. We may first write
	$$
	T_f(m,n;p^{\gamma} \div p^{2\beta+\gamma+\epsilon}) = e\left( \frac{mf+n\overline{f}}{p^{2\beta+\gamma+\epsilon}} \right) \sum_{\substack{a \mod(p^{2\beta+\epsilon})}} e\left(\frac{ma-n\overline{f}^2 a[1+\overline{f}ap^{\gamma}]^{-1}}{p^{2\beta+\epsilon}} \right).
	$$
	Let $h(a)=m-n\overline{f}^2(1+\overline{f}ap^{\gamma})^{-2}$ denote the derivative of the numerator in the exponential. We note that $p^{\ell\gamma} \ddiv \frac{1}{(\ell+1)!} h^{(\ell)}$. We shall write $a=u+vp^{\max\{0,\beta+\epsilon-\gamma\}}$ with $u \, \mod(p^{\max\{0,\beta+\epsilon-\gamma\}})$ and $v \, \mod(p^{\min\{2\beta+\epsilon, \beta+\gamma\}})$ and Taylor expand:
	\begin{multline*}
	ma-n\overline{f}^2 a[1+\overline{f}ap^{\gamma}]^{-1} \equiv mu-n\overline{f}^2 u[1+\overline{f}up^{\gamma}]^{-1} \\
	+h(u)vp^{\max\{0,\beta+\epsilon-\gamma\}}+(\tfrac{1}{2}h'(u))v^2p^{2\max\{0,\beta+\epsilon-\gamma\}} \quad \mod(p^{2\beta+\epsilon}).
	\end{multline*}
	The higher order terms vanish modulo $p^{2\beta+\epsilon}$ as $2 \gamma + \max\{0,\beta+\epsilon-\gamma\}\ge 2\beta+2\epsilon \ge 2\beta+\epsilon$. We conclude
	\begin{multline*}
	T_f(m,n;p^{\gamma} \div p^{2\beta+\gamma+\epsilon}) = e\left( \frac{mf+n\overline{f}}{p^{2\beta+\gamma+\epsilon}} \right) \sum_{\substack{u \mod(p^{\max\{0,\beta+\epsilon-\gamma\}})}} e\left(\frac{mu-n\overline{f}^2 u[1+\overline{f}up^{\gamma}]^{-1}}{p^{2\beta+\epsilon}} \right) \\
	\times \mathcal{G}((\tfrac{1}{2}h'(u)/p^{\gamma})p^{\max\{\beta+\epsilon,\gamma\}}, \,  h(u)  ;\, p^{\min\{2\beta+\epsilon,\beta+\gamma\}}).
	\end{multline*}
	Upon recalling $\beta \ge \gamma/2$, we find $p^{\max\{\beta+\epsilon,\gamma\}} \div p^{\min\{2\beta+\epsilon,\beta+\gamma\}}$. Hence, the Gau{\ss} sum vanishes unless $p^{\max\{\beta+\epsilon,\gamma\}} \div h(u)$ according to Lemma \ref{lem:Gauss}. Thus, 
	\begin{multline*}
	T_f(m,n;p^{\gamma} \div p^{2\beta+\gamma+\epsilon}) = e\left( \frac{mf+n\overline{f}}{p^{2\beta+\gamma+\epsilon}} \right) p^{\max\{\beta+\epsilon,\gamma\}} \\
	\times \sum_{\substack{u \mod(p^{\max\{0,\beta+\epsilon-\gamma\}}) \\ h(u) \equiv 0 \mod(p^{\max\{\beta+\epsilon,\gamma\}})}} e\left(\frac{mu-n\overline{f}^2 u[1+\overline{f}up^{\gamma}]^{-1}}{p^{2\beta+\epsilon}} \right) \\
	\times \mathcal{G}((\tfrac{1}{2}h'(u)/p^{\gamma}), \,  (h(u)/p^{\max\{\beta+\epsilon,\gamma\}})  ;\, p^{2\min\{\beta+\epsilon,\gamma\}-\gamma-\epsilon}).
	\end{multline*}	
	The number of solution $u \, \mod(p^{\max\{0,\beta+\epsilon-\gamma\}})$ to $h(u) \equiv 0 \, (p^{\max\{\beta+\epsilon,\gamma\}})$ is $0$ unless $f^2m \equiv n \, \mod(p^{\gamma})$ and otherwise bounded by $1$ if $p \neq 2$ and $4$ if $p=2$. Thus, we conclude
	$$
	|T_f(m,n;p^{\gamma} \div p^{2\beta+\gamma+\epsilon})| \le p^{\beta+\epsilon/2+\gamma/2} \cdot \begin{cases} 0, & p^{\gamma} \notdiv f^2m-n,\\ 1, & p\neq 2, \, p^{\gamma} \div f^2m-n, \\ 4,& p = 2, \, p^{\gamma} \div f^2m-n.  \end{cases} 
	$$
	The claimed bound now follows from multiplying together all of the local bounds.

\end{proof}

\section{Maa{\ss} forms} 
\label{sec:maassmain}

\subsection{Maa{\ss} forms and Eisenstein series}

\label{sec:Maassforms}

Let $j(\gamma,z)=cz+d$ for $\gamma = \sm a & b \\ c & d \esm \in \SL{2}(\mathbb{R})$ and $z \in \mathbb{H}$. For a function $f: \mathbb{H} \to \mathbb{C}$ and $\gamma \in \SL{2}(\mathbb{R})$, we set
$$
(f|_{\kappa}\gamma)(z) = \left( \frac{j(\gamma,z)}{|j(\gamma,z)|} \right)^{-\kappa} f(\gamma z),
$$
where $\SL{2}(\mathbb{R})$ acts on $\mathbb{H}$ via M\"obius transformations. We shall equip the upper half-plane $\mathbb{H}$ with the invariant measure $d\mu(z)=y^{-2}dxdy$, where $z=x+iy$ with $x,y$ real, and fix the inner product
\begin{equation}
\langle f,g \rangle = \int_{\Gamma \backslash \mathbb{H}} f(z) \overline{g(z)} d\mu(z).
\label{eq:innerprod}
\end{equation}

A \emph{Maa{\ss} form} of weight $\kappa \in \{0,1\}$ with respect to $\Gamma$ and multiplier system $\upsilon$ is a smooth function $f:\mathbb{H} \to \mathbb{C}$ satisfying the following conditions:
\begin{enumerate}
	\item \label{crit:1} $f|_{\kappa}\gamma = \upsilon(\gamma) f$ for all $\gamma \in \Gamma$,
	\item \label{crit:2} $-\Delta_{\kappa}f=\lambda_f f=(\frac{1}{4}+t_f^2)f$ is an eigenfunction of the weight-$\kappa$ Laplace--Beltrami operator $\Delta_{\kappa}=y^2(\frac{\partial^2}{\partial x^2}+\frac{\partial^2}{\partial y^2})-i\kappa y \frac{\partial}{\partial x}$,
	\item \label{crit:3} $f$ is of moderate growth, i.e.\@ for each cusp $\mathfrak{a}$ $(f|_{\kappa} \sigma_{\mathfrak{a}})(z)$ grows at most polynomially as $y \to \infty$ uniformly in $x$,
	\item \label{crit:4} $f$ has a finite $L^2$-norm with respect the inner product \eqref{eq:innerprod}.
\end{enumerate}
The eigenvalues $\lambda_f$ of a Maa{\ss} form are contained in $[\frac{\kappa}{2}(1-\frac{\kappa}{2}),\infty[$.
Maa{\ss} forms admit a Fourier expansion at each cusp $\mathfrak{a}$:
\begin{equation}
(f|_{\kappa}\sigma_{\mathfrak{a}})(z) = \sum_{m \in \mathbb{Z}} C_f(\mathfrak{a},m;y) e((m+\eta_{\mathfrak{a}})x),
\label{eq:Maass-Fourier}
\end{equation}
where
\begin{equation}
C_f(\mathfrak{a},m;y) = \begin{cases} \rho_f(\mathfrak{a},m) W_{\sign(m+\eta_{\mathfrak{a}})\frac{\kappa}{2},it_f}(4 \pi |m+\eta_{\mathfrak{a}}| y), & m+\eta_{\mathfrak{a}} \neq 0, \\
\rho_f(\mathfrak{a},0) y^{\frac{1}{2}+it_f}+\rho_f'(\mathfrak{a},0)y^{\frac{1}{2}-it_f}, & m=\eta_{\mathfrak{a}}=0, t_f \neq 0,\\
 \rho_f(\mathfrak{a},0) y^{\frac{1}{2}}+\rho_f'(\mathfrak{a},0)y^{\frac{1}{2}}\log(y), & m=\eta_{\mathfrak{a}}=t_f=0, \end{cases}
\label{eq:Gen-Fourier}
\end{equation}
and $ W_{k,m}$ is the Whittaker function.

For each singular cusp $\mathfrak{a}$, one can define \emph{Eisenstein series}
\begin{equation}
\mathrm{E}^{\upsilon,\kappa}_{\mathfrak{a}}(z,s) = \sum_{\gamma \in \widehat{\Gamma}_{\mathfrak{a}} \backslash \widehat{\Gamma}} \overline{\upsilon(\gamma) } \Im(\sigma_{\mathfrak{a}}^{-1} \gamma z )^{s} \left( \frac{j(\sigma_{\mathfrak{a}}^{-1} \gamma,z)}{|j(\sigma_{\mathfrak{a}}^{-1} \gamma,z)|}  \right)^{-\kappa}.
\label{eq:Eisenstein-def}
\end{equation}
They are \emph{a priori} only defined for $\Re(s) > 1$, but admit a meromorphic continuation to the complex plane. They satisfy the criteria \eqref{crit:1}-\eqref{crit:3} and admit a Fourier expansion at an arbitrary cusp $\mathfrak{b}$ (see \cite[Prop.\@ 3.4.6]{thesis}\footnote{The reference assumes that one chose a fixed representative for each $\Gamma$-equivalence class of cusps and a fixed scaling matrix for each representative. The relaxation to the consistency relation $\sigma_{\gamma \mathfrak{a}}=\gamma \sigma_{\mathfrak{a}}$ only affects the diagonal term stemming from \cite[Eq. 3.17]{thesis}.}):
\begin{multline}
	(\mathrm{E}^{\upsilon,\kappa}_{\mathfrak{a}}|_{\kappa}\sigma_{\mathfrak{b}})(z,s) = \overline{\upsilon(\sigma_{\mathfrak{a}}\sigma_{\mathfrak{b}}^{-1})} \delta_{\mathfrak{a} \equiv\mathfrak{b} \, (\Gamma)} y^s+\delta_{\eta_{\mathfrak{b}},0} 2^{2-2s} \pi \frac{\Gamma(2s-1)}{\Gamma(s-\frac{\kappa}{2})\Gamma(s+\frac{\kappa}{2})} \mathcal{Z}^{\upsilon,\kappa}_{\mathfrak{a},\mathfrak{b}}(0,0;s) y^{1-s} \\
	+ \pi^{s} \sum_{\substack{n \in \mathbb{Z} \\ n+\eta_{\mathfrak{b}} \neq 0}} \frac{|n + \eta_{\mathfrak{b}}|^{s-1}}{\Gamma(s+\sign(n+\eta_{\mathfrak{b}})\frac{\kappa}{2})} \mathcal{Z}^{\upsilon,\kappa}_{\mathfrak{a},\mathfrak{b}}(0,n;s) W_{\sign(n+\eta_{\mathfrak{b}})\frac{\kappa}{2},\frac{1}{2}-s}(4 \pi |n+\eta_{\mathfrak{b}}|y) e((n+\eta_{\mathfrak{b}})x),
	\label{eq:Eis-Fourier}
\end{multline}
where the Kloosterman Zeta function is given by
\begin{equation}
	\mathcal{Z}^{\upsilon,\kappa}_{\mathfrak{a},\mathfrak{b}}(m,n;s) = \sum_{c \in \mathcal{C}_{\mathfrak{a},\mathfrak{b}}} \frac{S^{\upsilon,\kappa}_{\mathfrak{a},\mathfrak{b}}(m,n;c)}{c^{2s}}.
	\label{eq:Kloos-Zeta}
\end{equation}
The Kloosterman Zeta function converges absolutely for $\Re(s)>1$ and admits a meromorphic continuation to $\Re(s)>\frac{1}{2}$ with at most simple poles located at $s=\frac{1}{2}+\sqrt{\frac{1}{4}-\lambda}$, where $\lambda \neq 0 $ is part of the spectrum \cite{SelbergFourier}. In analogy to the Fourier coefficients of Maa{\ss} forms \eqref{eq:Maass-Fourier},\eqref{eq:Gen-Fourier}, we may define Fourier coefficients $\rho_{\mathrm{E}}(\mathfrak{a},m;r)$ of an Eisenstein series $\mathrm{E}$ at $s=\frac{1}{2}+ir$. Explicitly, we have for $n+\eta_{\mathfrak{b}} \neq 0$:
\begin{equation}
	\rho_{\mathrm{E}^{\upsilon, \kappa}_{\mathfrak{a}}}(\mathfrak{b},n;r) = \pi^{\frac{1}{2}+ir} \frac{|n+\eta_{\mathfrak{b}}|^{-\frac{1}{2}+ir}}{\Gamma(\frac{1}{2}+ir + \sign(n+\eta_{\mathfrak{b}})\frac{\kappa}{2})} \mathcal{Z}^{\upsilon,\kappa}_{\mathfrak{a},\mathfrak{b}}(0,n;\tfrac{1}{2}+ir).
	\label{eq:Eis-Fourier-coeff}
\end{equation}
We shall denote the set of Eisenstein series $\mathrm{E}^{\upsilon,\kappa}_{\mathfrak{a}}$, where $\mathfrak{a}$ runs over a representative set of singular cusps, by $\mathcal{B}^{\kappa}_{\mathrm{E}}(\Gamma,\upsilon)$. Likewise, we denote by $\mathcal{B}_{\mathrm{M}}^{\kappa}(\Gamma,\upsilon)$ an orthonormal basis of Maa{\ss} forms for the discrete/non-continuous subspace.

The significance of these functions is that the Eisenstein series on the line $\Re(s)=\frac{1}{2}$ together with the Maa{\ss} forms form a complete eigenpacket for the $L^2$-space of functions satisying \eqref{crit:1}, \eqref{crit:3}, and \eqref{crit:4}. The latter space we shall denote $L^{2}(\Gamma,\upsilon,\kappa)$.

\subsection{Pre-trace formula}

\begin{prop}[Kuznetsov pre-trace formula] Let $\upsilon$ be a multiplier system of weight $\kappa \in \{0,1\}$ with respect to $\Gamma$, $\mathfrak{a},\mathfrak{b}$ two cusps of $\Gamma$, and $m,n$ two integers satisfying $(m+\eta_{\mathfrak{a}})(n+\eta_{\mathfrak{b}}) \neq 0$. Denote  $\epsilon = \frac{1}{2} (\sign(m+\eta_{\mathfrak{a}})+\sign(n+\eta_{\mathfrak{b}})) \in\{-1,0,1\}$. Suppose that the Kloosterman sums satisfy
	\begin{equation}\begin{gathered}
		\sum_{c \in \mathcal{C}_{\mathfrak{a},\mathfrak{b}}} \frac{|S^{\upsilon,\kappa}_{\mathfrak{a},\mathfrak{b}}(m,n;c)|}{c^2} |\log(c)| <\infty, \quad
		\sum_{c \in \mathcal{C}_{\mathfrak{a},\mathfrak{a}}} \frac{|S^{\upsilon,\kappa}_{\mathfrak{a},\mathfrak{a}}(m,m;c)|}{c^2} |\log(c)| <\infty,\\
		\sum_{c \in \mathcal{C}_{\mathfrak{b},\mathfrak{b}}} \frac{|S^{\upsilon,\kappa}_{\mathfrak{b},\mathfrak{b}}(n,n;c)|}{c^2} |\log(c)| <\infty. \label{eq:Kloos-assump}
	\end{gathered}\end{equation}
	Then, for any $t \in \mathbb{R}$, we have

	\begin{multline}
		\sum_{f \in \mathcal{B}_{\mathrm{M}}^{\kappa}(\Gamma,\upsilon)} \frac{\sqrt{|(m+\eta_{\mathfrak{a}})(n+\eta_{\mathfrak{b}})|} }{\cosh(\pi(t-t_f))\cosh(\pi(t+t_f))} \overline{\rho_f(\mathfrak{a},m)} \rho_f(\mathfrak{b},n)\\
		+ \frac{1}{4\pi} \sum_{\mathrm{E} \in \mathcal{B}^{\kappa}_{\mathrm{E}}(\Gamma,\upsilon)} \int_{-\infty}^{\infty} \frac{\sqrt{|(m+\eta_{\mathfrak{a}})(n+\eta_{\mathfrak{b}})|} }{\cosh(\pi(t-r))\cosh(\pi(t+r))} \overline{\rho_{\mathrm{E}}(\mathfrak{a},m;r)} \rho_{\mathrm{E}}(\mathfrak{b},n;r) dr \\
		= \frac{|\Gamma(1-\epsilon\frac{\kappa}{2}+it)|^2}{4\pi^3} \Biggl\{ \overline{\upsilon(\sigma_{\mathfrak{a}}\sigma_{\mathfrak{b}}^{-1})} |\epsilon|\delta_{\mathfrak{a}\equiv \mathfrak{b} \, (\Gamma)}\delta_{m,n} \\+ \sum_{c \in \mathcal{C}_{\mathfrak{a},\mathfrak{b}}} \frac{S^{\upsilon,\kappa}_{\mathfrak{a},\mathfrak{b}}(m,n;c)}{c} I_{\epsilon\kappa}^{\pm}\left( \frac{4 \pi \sqrt{|(m+\eta_{\mathfrak{a}})(n+\eta_{\mathfrak{b}})|}}{c} ,t \right)    \Biggr\},
		\label{eq:Kuz-pre-equal}
	\end{multline}
	where the sign $\pm$ is according to the sign of $(m+\eta_{\mathfrak{a}})(n+\eta_{\mathfrak{b}})$ and
	$$\begin{aligned}
		I_{\kappa}^+(\omega,t) &= -2i \omega \int_{-i}^{i} K_{2it}\left( \omega q \right) q^{\kappa-1} dq, \\
		I_{0}^{-}(\omega,t) &= 2 \omega \frac{\sinh(\pi t)}{t} K_{2it}(\omega).
	\end{aligned}$$
Here, the integral $\int_{-i}^{i}$ is along the unit circle in positive/anti-clockwise direction.

	\label{prop:Kuz-trace}
\end{prop}
\begin{proof} This is \cite[Prop. 3.6.8--3.6.10]{thesis}\footnote{The reference assumes that one chose a fixed representative for each $\Gamma$-equivalence class of cusps and a fixed scaling matrix for each representative. The relaxation to the consistency relation $\sigma_{\gamma \mathfrak{a}}=\gamma \sigma_{\mathfrak{a}}$ only affects the diagonal term stemming from \cite[Eq. 3.17]{thesis}.} together with Remark \ref{rem:Kloos-conj}. The extra assumption on the Kloosterman sums guarantees the absolute convergence of both sides as $\sigma \to 1^+$. One may also compare \cite[Lem 3]{Proskurin}, \cite[Eq. (4.50)]{Kuz1}, \cite[Prop. 5.2]{DFI2}, or \cite[\S4]{IwDes}.
\end{proof}

\subsection{Bounds on Fourier coefficients} \label{sec:twisted-Fourier-Maass} Let $\kappa \in \{0,1\}$, $q,Q \in \mathbb{N}$ with $q \div Q$. The space $L^2(\Gamma_{0,\pm 1}(Q;Q/q), \sign_{Q/q}^{\kappa}, \kappa )$ is the orthogonal direct sum of the spaces $L^2(\Gamma_0(Q), \chi, \kappa)$ where $\chi$ runs over all Dirichlet characters of modulus $Q/q$ and $\chi(-1)=(-1)^{\kappa}$. We note that the index satisfies
$$
I_{Q/q} := [\Gamma_0(Q):\Gamma_{0,\pm 1}(Q;Q/q)] = \varphi(Q/q) \cdot \begin{cases}1, & Q/q=1,2, \\ 1/2, & \text{otherwise},  \end{cases} = \left(\tfrac{Q}{q}\right)^{1+o(1)}.
$$
Thus, we may choose eigenpackets as follows
\begin{equation}\begin{aligned}
	\mathcal{B}^{\kappa}_{\mathrm{M}}(\Gamma_{0,\pm 1}(Q;Q/q), \sign_{Q/q}^{\kappa}) &=  (I_{Q/q})^{-1/2} \bigsqcup_{\substack{\chi \mod Q/q \\ \chi(-1)=(-1)^{\kappa}}} \mathcal{B}^{\kappa}_{\mathrm{M}}(\Gamma_0(Q),\chi), \\
	\mathcal{B}^{\kappa}_{\mathrm{E}}(\Gamma_{0,\pm 1}(Q;Q/q), \sign_{Q/q}^{\kappa}) &= (I_{Q/q})^{-1/2} \bigsqcup_{\substack{\chi \mod Q/q \\ \chi(-1)=(-1)^{\kappa}}} \mathcal{B}^{\kappa}_{\mathrm{E}}(\Gamma_0(Q),\chi).
\end{aligned} \label{eq:basis-choice}\end{equation}

A standard application of the pre-trace formula (Prop. \ref{prop:Kuz-trace}) is to demontrate an asymptotic orthogonality relation between the Fourier coeffients. Here, we only require an upper bound.

\begin{prop} Let $\kappa \in \{0,1\}$, $q,Q \in \mathbb{N}$ with $q \mid Q$, $0 \neq m \in \mathbb{Z}$, and $\epsilon = \sign(m)$. Then, 
	\begin{multline*}
		\sum_{\substack{\chi \mod Q/q \\ \chi(-1)=(-1)^{\kappa}}} \Biggl(  \sum_{\substack{f \in \mathcal{B}_{\mathrm{M}}^{\kappa}(\Gamma_0(Q),\chi) \\ |t_f| \le T}} \frac{|m|}{\cosh(\pi t_f)} (1+|t_f|)^{\epsilon\kappa} |\rho_f(\infty,m)|^2 \\
		+ \sum_{\mathrm{E} \in \mathcal{B}^\kappa_{\mathrm{E}}(\Gamma_0(Q),\chi)} \int_{-T}^{T} \frac{|m|}{\cosh(\pi t)} (1+|t|)^{\epsilon\kappa} |\rho_{\mathrm{E}}(\infty,m;t)|^2 dt   \Biggr) \\
		\ll \left(\frac{Q}{q}\right)^{1+o(1)} \left(T^2 +  \frac{|m|^{1/2+o(1)}(q,m)^{1/2}}{Q^{1-o(1)}} \right)
	\end{multline*}
	and
	\begin{multline*}
		\sum_{\substack{\chi \mod Q/q \\ \chi(-1)=(-1)^{\kappa}}} \sum_{\substack{r \mod((q,Q/q))\\ (r,Q)=1}} \Biggl( \sum_{\substack{f \in \mathcal{B}_{\mathrm{M}}^{\kappa}(\Gamma_0(Q),\chi) \\ |t_f| \le T}} \frac{|mw_q|}{\cosh(\pi t_f)} (1+|t_f|)^{\epsilon\kappa} |\rho_f(r/q,mw_q)|^2 \\
		+ \sum_{\mathrm{E} \in \mathcal{B}^\kappa_{\mathrm{E}}(\Gamma_0(Q),\chi)} \int_{-T}^{T} \frac{|mw_q|}{\cosh(\pi t)} (1+|t|)^{\epsilon\kappa} |\rho_{\mathrm{E}}(r/q,mw_q;t)|^2 dt \Biggr) \\
		\ll \left(\frac{Q}{q}\right)^{1+o(1)} (q,Q/q) \left(T^2 +  \frac{|m|^{1/2+o(1)}(q,m)^{1/2}}{Q^{1-o(1)}}  \right).
	\end{multline*}
	\label{prop:chi-avg-fourier-bound}
\end{prop}
\begin{proof} We apply Proposition \ref{prop:Kuz-trace} for the group $\Gamma_{0,\pm 1}(Q;Q/q)$ with the multiplier system $\sign_{Q/q}^{\kappa}$, equal cusps and Fourier coefficients, and the choice of basis given by \eqref{eq:basis-choice}. Concretely, we sum the equality \eqref{eq:Kuz-pre-equal} over the relevant cusps and multiply it with $2 \pi t |\Gamma(1-\epsilon\frac{\kappa}{2}+it)|^{-2}$ and integrate $t$ from $0$ to $T$. On the spectral side, we use that
	$$\begin{aligned}
	H_{\epsilon\kappa}(r,T) &= 2\pi \int_0^T \frac{t}{|\Gamma(1-\epsilon\frac{\kappa}{2}+it)|^2} \frac{\cosh(\pi r)}{\cosh(\pi(t+r))\cosh(\pi(t-r))} dt \\
	&= 2\pi \int_0^T \frac{t}{|\Gamma(1-\epsilon\frac{\kappa}{2}+it)|^2} \frac{\cosh(\pi r)}{\sinh(\pi t)^2+\cosh(\pi r)^2} dt
	\end{aligned}$$
	is non-negative for $r \in \mathbb{R} \cup ]-\frac{1}{2}i,\frac{1}{2}i[$ and satisfies
	\begin{equation}
	H_{\epsilon\kappa}(r,T) \gg_{\delta} (1+|r|)^{\epsilon\kappa}
	\label{eq:H-lower}
	\end{equation}
	for $r \in \mathbb{R} \cup ]-(\frac{1}{2}-\delta)i,(\frac{1}{2}-\delta)i[$ with $|r|\le T$ together with a spectral gap, e.g\@ the Selberg spectral gap $|\Im t_j| \le \frac{1}{4}$ \cite{SelbergFourier}. The Inequality \eqref{eq:H-lower} itself follows from looking at the interval $[\frac{1}{2},1]$ if $|r|\le 1$, respectively $[r-\frac{1}{2},r]$ if $1\le r\le T$, and Stirling's approximation for the Gamma function. On the geometric side, we collect the Kloosterman sums according to Lemma \ref{lem:Kloos-Gamma01-Q-Q/q} and use
	$$
	\int_0^{T} t I_{\epsilon\kappa}^+(\omega,t) dt \ll \begin{cases} \omega^{1/2}, & \omega \ge 1,\\ \omega(1+|\log(\omega)|), & \omega \le 1, \end{cases}
	$$
	see \cite[Eq. (5.15)]{Kuz1} for the case $\kappa=0$ and \cite[Lem 6.2]{Humphries16} for $\kappa=1$. The diagonal term thus contributes $O(T^2)$, respectively $O((q,Q/q)T^2)$, and the terms involving the Kloosterman sums we bound using Corollary \ref{cor:bound-kloosterman-after-CS} as follows:
	\begin{multline*}
		\sum_{c \in Q \mathbb{N}} \frac{|S_{\infty,\infty}^{\Gamma_{0,\pm 1}(Q,Q/q),\sign_{Q/q}^{\kappa},\kappa}(m,m;c)|}{c}  \times \min\left\{  \left( \frac{|m|}{c} \right)^{1/2}  , \left( \frac{|m|}{c} \right)^{1-o(1)} \right\}  \\
		\ll \frac{q^{1/2}(q,m)^{1/2}}{Q^{1-o(1)}} \sum_{\substack{e \mid m/(q,m)}} \frac{1}{e^{1/2-o(1)}} \sum_{\substack{c=1}}^{\infty} \frac{\min\{\frac{q}{(q,m)}c, \frac{Q}{q}\}^{1/2}}{c^{1/2-o(1)}} \times \min \left\{ \left(\frac{|m|}{Qc}\right)^{1/2}, \left(\frac{|m|}{Qc}\right)^{1-o(1)}\right\} \\
		\ll \frac{|m|^{1+o(1)}(q,m)^{1/2}}{Q^{3/2-o(1)}} \min\left\{ 1, \frac{q^2}{Q(q,m)}, \frac{Q}{|m|} \right\}^{1/2}
		\label{eq:diag-infty-Kloostersum-bound}
	\end{multline*}
	and
	\begin{align*} & \sum_{cqw_q \in Q \mathbb{N}}  \frac{1}{cqw_q} \left| \sum_{r \mod((q,Q/q))}   S_{r/q,r/q}^{\Gamma_{0,\pm 1}(Q,Q/q),\sign_{Q/q}^{\kappa},\kappa}(mw_q,mw_q;cqw_q) \right|   \\ & \hspace{7cm}  {  \times \min\left\{ \left( \frac{|m|}{cq} \right)^{1/2}  , \left( \frac{|m|}{cq} \right)^{1-o(1)} \right\}}  \\
		& \hspace{5mm} \ll  \frac{(q,Q/q)}{Q^{1-o(1)}} \sum_{e \mid w_q} e^{o(1)} \sum_{f \mid m/(q,m)} \frac{1}{f^{1/2-o(1)}} \sum_{c=1}^{\infty} \frac{q^{1/2}(q,m)^{1/2}}{c^{1/2-o(1)}} \min\left\{ \frac{cq}{(q,m)}, (q,Q/q)  \right\}^{1/2} \\ &  \hspace{7cm} { \times \min \left\{ \left( \frac{|m|}{c(q^2,Q)} \right)^{1/2} , \left( \frac{|m|}{c(q^2,Q)} \right)^{1-o(1)} \right\} } \\
		& \hspace{5mm} \ll  \frac{|m|^{1+o(1)}(q,Q/q)^{1/2}(q,m)^{1/2}}{q^{1/2}Q^{1-o(1)}} \min \left\{ 1, \frac{q}{(q,Q/q)(q,m)} , \frac{q (q,Q/q)}{|m|}  \right\}^{1/2}.
 \end{align*}
\end{proof}

\subsection{Density estimates for exceptional eigenfunctions}

Retain the notation from \S\ref{sec:twisted-Fourier-Maass}. Exceptional eigenfunctions $f$, i.e.\@ those with $\Im t_f \neq 0$, may only arise if $\kappa =0$ since for $\kappa=1$ the $-\Delta_1$ eigenvalues are trivially $\ge \frac{1}{4}$, see \S\ref{sec:Maassforms}. In bounding their Fourier coefficients, one may amplify the contribution of those which are more exceptional. We have the following bound.

\begin{prop} Let $q,Q \in \mathbb{N}$ with $q \mid Q$ and $0 \neq m \in \mathbb{Z}$. Then, for $X\ge 1$, we have 
	\begin{multline*}
		\sum_{\substack{\chi \mod Q/q \\ \chi(-1)=1}}  \sum_{\substack{f \in \mathcal{B}_{\mathrm{M}}^{0}(\Gamma_0(Q),\chi) \\ 0<|\Im t_f| \le \frac{1}{2}}} \frac{|m|}{\cosh(\pi t_f)} |\rho_f(\infty,m)|^2 X^{2|t_f|} \\
		\ll \frac{Q}{q} \left(1+ \frac{|m|^{1/2}(q,m)^{1/2}X^{1/2}}{Q} \right) |mQX|^{o(1)}
	\end{multline*}
	and
	\begin{multline*}
		\sum_{\substack{\chi \mod Q/q \\ \chi(-1)=1}} \sum_{\substack{r \mod((q,Q/q))\\(r,Q)=1}} \sum_{\substack{f \in \mathcal{B}_{\mathrm{M}}^{0}(\Gamma_0(Q),\chi) \\ 0<|\Im t_f| \le \frac{1}{2}}} \frac{|mw_q|}{\cosh(\pi t_f)}  |\rho_f(r/q,mw_q)|^2 X^{2|t_f|} \\
		\ll (q,Q/q)\frac{Q}{q}\left( 1+\frac{|m|^{1/2}(q,m)^{1/2}X^{1/2}}{Q} \right)|mQX|^{o(1)}.
	\end{multline*}
	\label{prop:exceptional-chi-avg-fourier-bound}
\end{prop}


\begin{proof} This follows from the argument on \cite[p. 414, 415]{ANT} or \cite[p. 169]{IwSpecMeth}. One has to simply substitute the bound for the sum of Kloosterman sums given in the proof of Proposition \ref{prop:chi-avg-fourier-bound}, i.e.\@
	\begin{equation*}
		\sum_{c \in \mathcal{C}_{\infty,\infty}} \frac{|S_{\infty,\infty}^{\Gamma_{0,\pm1}(Q;Q/q),1,0}(m,m;c)|}{c} \left( \frac{|m|}{c} \right)^{\frac{1}{2}+o(1)} \ll \frac{|m|^{1/2}(q,m)^{1/2}}{Q} |mQ|^{o(1)},
	\end{equation*}
	and
	\begin{multline*}
		\sum_{\substack{r \mod((q,Q/q))\\ (r,Q)=1}}\sum_{c \in \mathcal{C}_{r/q,r/q}} \frac{|S_{r/q,r/q}^{\Gamma_{0,\pm1}(Q;Q/q),1,0}(mw_q,mw_q;c)|}{c} \left( \frac{|mw_q|}{c} \right)^{\frac{1}{2}+o(1)}  \\ \ll (q,Q/q) \frac{|m|^{1/2}(q,m)^{1/2}}{Q} |mQ|^{o(1)}.
	\end{multline*}
\end{proof}

\section{Holomorphic forms} \label{sec:holmain}

Let $k \in \mathbb{N}$. Similar to the Maa{\ss} case, we introduce the slash operator\footnote{No confusion should arise from this abuse of notation.} for $z \in \mathbb{H}$ and $\gamma \in \SL{2}(\mathbb{R})$
$$
(f|_k \gamma)(z) = j(\gamma,z)^{-k}f(\gamma z)
$$
and the Petersson inner product
\begin{equation}
	\langle f,g \rangle = \int_{\Gamma \backslash \mathbb{H}} f(z) \overline{g(z)} y^k d\mu(z).
	\label{eq:Petersson-Inner}
\end{equation}
A function $f:\mathbb{H}\to \mathbb{C}$ is said to be a \emph{holomorphic cusp form} of weight $k$ with respect to $\Gamma$ and multiplier system $\upsilon$ if it satisfies the following properties:
\begin{enumerate}
	\item $f$ is holomorphic,
	\item $f|_k\gamma = \upsilon(\gamma) f$ for all $\gamma \in \Gamma$,
	\item $f$ is holomorphic at all cusps of $\Gamma$,
	\item $f$ has finite $L^2$-norm with respect to the inner product \eqref{eq:Petersson-Inner}.
\end{enumerate}
For each cusp $\mathfrak{a}$ of $\Gamma$, $L^2$-finite holomorphic forms admit a Fourier expansion
\begin{equation}
	(f|_k\sigma_{\mathfrak{a}})(z) = \sum_{\substack{m \in \mathbb{Z}\\ m+\eta_{\mathfrak{a}} > 0}} \psi_f(\mathfrak{a},m) e((m+\eta_{\mathfrak{a}})z).
	\label{eq:holo-Fourier}
\end{equation}
The space of holomorphic cusp forms of a given weight $k$, Fuchsian group $\Gamma$, and multiplier system $\upsilon$ is finite and we shall denote by $\mathcal{B}^k_{\rm H}(\Gamma,\upsilon)$ an orthonormal basis thereof.
%

\subsection{Petersson trace formula}

\begin{prop}[Petersson trace formula]
	Let $\upsilon$ be a multiplier system of weight $\kappa \in \{0,1\}$ with respect to $\Gamma$. Let $k \ge 2$ be an integer with $k \equiv \kappa \, (2)$, $\mathfrak{a},\mathfrak{b}$ two cusps of $\Gamma$, and $m,n$ two integers satisfying $m+\eta_{\mathfrak{a}}, n+\eta_{\mathfrak{b}} > 0$.
	Then, we have
	\begin{multline*}
		\frac{\Gamma(k-1)}{(4 \pi \sqrt{(m+\eta_{\mathfrak{a}})(n+\eta_{\mathfrak{b}})})^{k-1}} \sum_{f \in \mathcal{B}_{\rm H}^k(\Gamma,\upsilon)} \overline{\psi_f(\mathfrak{a},m)} \psi_f(\mathfrak{b},n) \\
		= \overline{\upsilon(\sigma_{\mathfrak{a}}\sigma_{\mathfrak{b}}^{-1})} \delta_{\mathfrak{a} \equiv \mathfrak{b} \, (\Gamma)}\delta_{m,n} \\
		+ 2 \pi (-1)^{\frac{k-\kappa}{2}} \sum_{c \in \mathcal{C}_{\mathfrak{a},\mathfrak{b}}} \frac{S^{\upsilon,\kappa}_{\mathfrak{a},\mathfrak{b}}(m,n;c)}{c} J_{k-1}\left( \frac{4 \pi \sqrt{(m+\eta_{\mathfrak{a}})(n+\eta_{\mathfrak{b}})}}{c}\right),
	\end{multline*}
	where for $k=2$ the sum on the right-hand side is to be interpreted as the limit
	$$
	\lim_{\sigma\to 1^+} \sum_{c \in \mathcal{C}_{\mathfrak{a},\mathfrak{b}}} \frac{S^{\upsilon,\kappa}_{\mathfrak{a},\mathfrak{b}}(m,n;c)}{c^{2\sigma-1}} J_{k-1}\left( \frac{4 \pi \sqrt{(m+\eta_{\mathfrak{a}})(n+\eta_{\mathfrak{b}})}}{c}\right).
	$$
	\label{prop:Peterssontrace}
\end{prop}
\begin{proof} This is \cite[Thm 3.7.6]{thesis}. Compare also \cite[Thm 9.6]{IwSpecMeth} and \cite[\S 5]{MFaF}.
\end{proof}

\subsection{Bounds on Fourier coefficients} Recall the notation from \ref{sec:twisted-Fourier-Maass}. 

\begin{prop} Let $\kappa \in \{0,1\}$, $a,q,Q,m, \in \mathbb{N}$ with $q \mid Q$ and $(a,Q/q)=1$. Then, for any $K \in \mathbb{N}$, we have
	\begin{multline*}
	\sum_{k=1}^{K} \frac{\Gamma(\kappa+2k)}{(4 \pi m)^{\kappa+2k-1}} \sum_{\substack{\chi \mod Q/q \\ \chi(-1)=(-1)^{\kappa}}} \sum_{f \in \mathcal{B}_{\rm H}^{\kappa+2k}(\Gamma_0(Q),\chi)} |\psi_f(\infty,m)|^2 \\
	\ll
	 \left(\frac{Q}{q}\right)^{1+o(1)}\left(K^2+ \frac{m^{2/3+o(1)}(q,m)^{1/2}}{Q^{7/6-o(1)}} \right)
	\end{multline*}
	and
	\begin{multline*}
		\sum_{k=1}^{K} \frac{\Gamma(\kappa+2k)}{(4 \pi m w_q)^{\kappa+2k-1}} \sum_{\substack{\chi \mod Q/q \\ \chi(-1)=(-1)^{\kappa}}} \sum_{\substack{r \mod((q,Q/q))\\ (r,Q)=1}}  \sum_{\substack{f \in \mathcal{B}_{\rm H}^{\kappa+2k}(\Gamma_0(Q),\chi)}} |\psi_f(r/q,mw_q)|^2 \\
		\ll \left(\frac{Q}{q}\right)^{1+o(1)} (q,Q/q) \left(K^2 +  \frac{m^{2/3+o(1)}(q,m)^{1/2}}{(q^2,Q)^{1/6}Q^{1-o(1)}}  \right).
	\end{multline*}
	\label{prop:Fourier-holo-diag-upper-bound}
\end{prop}
\begin{proof} The proof is similar to the proof of Proposition \ref{prop:chi-avg-fourier-bound}. We make use of the Petersson trace formula (Prop. \ref{prop:Peterssontrace}) for the group $\Gamma_{0,\pm 1}(Q,Q/q)$ with multiplier system $\sign_{Q/q}^{\kappa}$, equal cusps and Fourier coefficients, and summed up over the relevant cusps. We further multiply the equality with $k-1$ and sum up over $k$ in order to match the spectral expression in the Proposition to be proven. On the geometric side, we note that
	$$\begin{aligned}
		\sum_{k=1}^{K} (-1)^k (\kappa+2k-1) J_{\kappa+2k-1}(\omega) &= \sum_{k=1}^{K} (-1)^k \frac{\omega}{2} \left( J_{\kappa+2k-2}(\omega)+J_{\kappa+2k}(\omega) \right) \\
		&= \frac{\omega}{2} \left( (-1)^{K}J_{\kappa+2K}(\omega)    - J_{\kappa}(\omega) \right).
	\end{aligned}$$
We shall make use of the following bound for the $J$-Bessel function valid for $k,\omega>0$
$$
|J_k(\omega)| \ll \min\{1,\omega^{-1/3}\},
$$
which follows from the two inequalities
\begin{align*}
	|J_{k}(\omega)| &\ll (\tfrac{\omega}{2})^k (k+2)^{\frac{1}{2}-k} \exp(\tfrac{1}{4}\omega^2), \\
|J_{k}(\omega)| &\ll \omega^{-\frac{1}{4}}(|\omega-k|+k^{\frac{1}{3}})^{-\frac{1}{4}}.
\end{align*}
The former follows from the Taylor expansion and the latter is recorded in \cite[Eq. 2.11']{lowlying} and is the combination of various upper bounds on the Bessel function in various ranges (cf. \cite{ToBF}). Finally, we also make use of Corollary \ref{cor:bound-kloosterman-after-CS} to bound the term involving the Kloosterman sums. We find
	\begin{align*}
	&\sum_{\substack{c \in Q \mathbb{N}}} \frac{|S_{\infty,\infty}^{\sign_{Q/q}^{\kappa},\kappa}(m,m;c)|}{c}  \times \min\left\{ \left( \frac{m}{c} \right),  \left( \frac{m}{c} \right)^{2/3}  \right\} \\
	&\hspace{5mm}\ll \frac{q^{1/2}(q,m)^{1/2}}{Q^{1-o(1)}} \sum_{\substack{e \mid m/(q,m)}} \frac{1}{e^{1/2-o(1)}} \sum_{\substack{c =1}}^{\infty} \frac{\min\{\frac{q}{(q,m)}c, \frac{Q}{q}\}^{1/2}}{c^{1/2-o(1)}} \min\left\{  \left( \frac{m}{Qc} \right),  \left( \frac{m}{Qc} \right)^{2/3}  \right\} \\
	&\hspace{5mm}\ll \frac{m^{1+o(1)}(q,m)^{1/2}}{Q^{3/2-o(1)}} \min\left\{ 1, \left(\frac{q^2}{Q(q,m)} \right)^{1/2}, \left(\frac{Q}{m}\right)^{1/3} \right\}.
\end{align*}
and
\begin{align*} & \sum_{cqw_q \in Q \mathbb{N}} \frac{1}{cqw_q} \left| \sum_{r \mod((q,Q/q))} S_{r/q,r/q}^{\sign_{Q/q}^{\kappa},\kappa}(mw_q,mw_q;cqw_q) \right| \times \min\left\{ \left( \frac{m}{cq} \right)  , \left( \frac{m}{cq} \right)^{2/3} \right\} \\
	& \hspace{5mm}\ll \frac{(q,Q/q)}{Q^{1-o(1)}} \sum_{e \mid w_q} e^{o(1)} \sum_{f \mid m/(q,m)} \frac{1}{f^{1/2-o(1)}} \sum_{c=1}^{\infty} \frac{q^{1/2}(q,m)^{1/2}}{c^{1/2-o(1)}} \min\left\{ \frac{cq}{(q,m)}, (q,Q/q)  \right\}^{1/2} \\ & \hspace{7cm}  \times \min \left\{ \left( \frac{m}{c(q^2,Q)} \right) , \left( \frac{m}{c(q^2,Q)} \right)^{2/3} \right\} \\
	& \hspace{5mm} \ll \frac{|m|^{1+o(1)}(q,Q/q)^{1/2}(q,m)^{1/2}}{q^{1/2}Q^{1-o(1)}} \min \left\{ 1, \left(\frac{q}{(q,Q/q)(q,m)}\right)^{1/2} ,\left( \frac{q (q,Q/q)}{m} \right)^{1/3}  \right\}.
\end{align*}
\end{proof}

\section{Kloosterman sums in arithmetic progression} \label{sec:final}

Let $a,q,Q \in \mathbb{N}$ be integers satisfying $q \mid Q$, $(a,Q/q)=1$. Let $m,n \in \mathbb{Z}\backslash\{0\}$ be two further non-zero integers. We wish to bound
\begin{equation}
\sum_{\substack{c \le C \\ c \equiv aq \mod(Q)}} \frac{S(m,n;c)}{c}.
\label{eq:Kloosterman-sum-sharp-to-bound}
\end{equation}
In an attempt to improve later displays, we shall introduce the notation
$$
F \prec G \quad \Leftrightarrow \quad  F \ll |mnQC|^{o(1)} G.
$$
Since $S(m,n;c)=S(-m,-n;c)$, we may and shall assume that $m>0$. Using the Weil bound \eqref{eq:Ex-Weilbound}, we are able to bound \eqref{eq:Kloosterman-sum-sharp-to-bound} by
\begin{equation}\begin{aligned}
		\sum_{\substack{c \le C \\ c \equiv aq \mod(Q)}} \frac{|S(m,n;c)|}{c} &\prec \frac{(m,n,q)^{1/2}}{q^{1/2}} \sum_{\substack{e \mid (m,n)/(m,n,q) \\ (e,Q/q)=1}} e^{1/2} \sum_{\substack{c \le C/q \\ c \equiv a \mod(Q/q) \\ e \mid c}} \frac{1}{c^{1/2}} \\
		& \prec (m,n,q)^{1/2} \left( \frac{1}{q^{1/2}}+\frac{C^{1/2}}{Q} \right).
		\label{eq:Klooster-AP-trivial}
\end{aligned}\end{equation}

In order to improve upon this for large values of $C$, we shall compare a dyadic sharp sum with a smooth one and apply the Bruggeman--Kuznetsov trace formula to relate the sum of Kloosterman sum to a spectral sum. The following version is recorded in \cite[Thms 3.10.1-3.10.3]{thesis}, compare also \cite{Proskurin} \cite[\S 2.1.4]{MR2382859}, and \cite[Thms 9.4 \& 9.8]{IwSpecMeth}.

\begin{prop}[Bruggeman--Kuznetsov trace formula]
	Let $\upsilon$ be a multiplier system of weight $\kappa \in \{0,1\}$ with respect to $\Gamma$, $\mathfrak{a},\mathfrak{b}$ two cusps of $\Gamma$, and $m,n$ two integers satisfying $(m+\eta_{\mathfrak{a}})(n+\eta_{\mathfrak{b}}) \neq 0$. Suppose  $\epsilon = \frac{1}{2} (\sign(m+\eta_{\mathfrak{a}})+\sign(n+\eta_{\mathfrak{b}})) \in\{0,1\}$. Further let $\phi : \mathbb{R}^+_0 \to \mathbb{C}$ be a function with continuous derivatives up to third order satisfying
	$$
	\phi(0)=\phi'(0)=0, \quad \phi(x) \ll (x+1)^{-1-\delta}, \quad \phi'(x),\phi''(x), \phi'''(x) \ll (x+1)^{-2-\delta}
	$$
	for some $\delta >0$. Then, we have
	\begin{multline*}
		\sum_{c \in \mathcal{C}_{\mathfrak{a},\mathfrak{b}}} \frac{S^{\upsilon,\kappa}_{\mathfrak{a},\mathfrak{b}}(m,n;c)}{c} \phi\left( \frac{4 \pi \sqrt{|(m+\eta_{\mathfrak{a}})(n+\eta_{\mathfrak{b}})|}}{c} \right) \\= \epsilon\mathcal{H}^{\upsilon,\kappa}_{\mathfrak{a},\mathfrak{b}}(m,n;\phi)+ \mathcal{M}^{\upsilon,\kappa,\epsilon,\pm}_{\mathfrak{a},\mathfrak{b}}(m,n;\phi)+\mathcal{E}^{\upsilon,\kappa,\epsilon,\pm}_{\mathfrak{a},\mathfrak{b}}(m,n;\phi),
	\end{multline*}
	where sign $\pm$ is according to the sign of $(m+\eta_{\mathfrak{a}})(n+\eta_{\mathfrak{b}})$,
	$$\begin{aligned}
		\mathcal{H}^{\upsilon,\kappa}_{\mathfrak{a},\mathfrak{b}}(m,n;\phi) &= \frac{1}{\pi} \sum_{\substack{k \equiv \kappa \,(2)\\ k \ge 2}} \sum_{f \in \mathcal{B}_{\rm H}^k(\Gamma,\upsilon)} \frac{(-1)^{\frac{k-\kappa}{2}}\Gamma(k)}{(4\pi \sqrt{|(m+\eta_{\mathfrak{a}})(n+\eta_{\mathfrak{b}})|})^{k-1}} \overline{\psi_f(\mathfrak{a},m)}\psi_f(\mathfrak{b},n) \widetilde{\phi}(k-1),\\
		\mathcal{M}^{\upsilon,\kappa,\epsilon,\pm}_{\mathfrak{a},\mathfrak{b}}(m,n;\phi) &= 4 \sum_{f \in \mathcal{B}_\mathrm{M}^{\kappa}(\Gamma,\upsilon)} \frac{\sqrt{|(m+\eta_{\mathfrak{a}})(n+\eta_{\mathfrak{b}})|}}{\cosh(\pi t_f)} \overline{\rho_f(\mathfrak{a},m)} \rho_f(\mathfrak{b},n) \widehat{\phi}^{\pm}_{\epsilon\kappa}(t_f),\\
		\mathcal{E}^{\upsilon,\kappa,\epsilon,\pm}_{\mathfrak{a},\mathfrak{b}}(m,n;\phi) 
		&=\frac{1}{\pi} \sum_{\mathrm{E} \in \mathcal{B}_{\mathrm{E}}^{\kappa}(\Gamma,\upsilon)} \int_{-\infty}^{\infty} \frac{\sqrt{|(m+\eta_{\mathfrak{a}})(n+\eta_{\mathfrak{b}})|}}{\cosh(\pi r)} \overline{\rho_{\mathrm{E}}(\mathfrak{a},m;r)} \rho_{\mathrm{E}}(\mathfrak{b},n;r) \widehat{\phi}^{\pm}_{\epsilon\kappa}(r) dr,
	\end{aligned}
	$$
	and
	$$\begin{aligned}
		\widetilde{\phi}(t) &= \int_0^{\infty} J_t(x) \phi(x) \frac{dx}{x}, \\
		\widehat{\phi}^{+}_{\kappa}(t) &= \frac{\pi i}{2} \frac{(-it)^{\kappa}}{\sinh(\pi t)} \int_0^{\infty} \left[ J_{2it}(x)-(-1)^{\kappa}J_{-2it}(x) \right]\phi(x) \frac{dx}{x}, \\
		\widehat{\phi}^{-}_{0}(t) &= 2 \cosh(\pi t) \int_0^{\infty} K_{2it}(x) \phi(x) \frac{dx}{x}.
	\end{aligned}
	$$
	\label{prop:Kuz-Brugge}
\end{prop}

Let $\phi: \mathbb{R}^+ \to \mathbb{R}$ be a smooth compactly supported bump function satisfying
\begin{enumerate}
	\item $\phi(x)=1$ for $2\pi \sqrt{|mn|}/C \le x \le 4\pi \sqrt{|mn|}/C$,
	\item $\phi(x)=0$ if either $x \le 2 \pi \sqrt{|mn|}/(C+B)$ or $x \ge 4 \pi \sqrt{|mn|}/(C-B)$,
	\item $\|\phi '\|_1 \ll 1$ and $\|\phi ''\|_1 \ll C/(XB)$,
\end{enumerate}
where $X=4\pi \sqrt{|mn|}/C$ and $1 \le B \le C/2$. The difference between the dyadic sharp and smooth sum
$$
\sum_{\substack{C  <  c \le 2C \\ c \equiv aq \mod(Q)}} \frac{S(m,n;c)}{c}-\sum_{c \equiv aq \mod(Q)} \frac{S(m,n;c)}{c} \phi\left(\frac{4 \pi \sqrt{|mn|}}{c}\right)
$$
may then be bounded using the Weil bound \eqref{eq:Ex-Weilbound}:
\begin{equation}\begin{aligned}
	\sum_{\substack{C \le c < C+A \\ c \equiv aq \mod(Q)}} \frac{|S(m,n;c)|}{c} &\prec \frac{(m,n,q)^{1/2}}{q^{1/2}} \sum_{\substack{e \mid (m,n)/(m,n,q) \\ (e,Q/q)=1}} e^{1/2} \sum_{\substack{C/q \le c < (C +A)/q \\ c \equiv a \mod(Q/q) \\ e \mid c}} \frac{1}{c^{1/2}} \\
	& \prec \min\left\{ \frac{(m,n,q)^{1/2}}{q^{1/2}} \left(1+\frac{A}{Q} \right) , \frac{(m,n)}{C^{1/2}}+\frac{(m,n,q)^{1/2}A}{QC^{1/2}}  \right\},
	\label{eq:Klooster-AP-trivial-short}
\end{aligned}\end{equation}
where $1\le A \le C$. We now bring the dyadic smooth sum into a shape where we may apply the Bruggeman--Kuznetsov trace formula. We proceed by making use of the explicit computations of Kloosterman sums in Lemma \ref{lem:Gamma0Q-Kloosterman}.

\begin{equation}\begin{aligned}
	&\sum_{c \equiv aq \mod(Q)} \frac{S(m,n;c)}{c} \phi\left( \frac{4 \pi \sqrt{|mn|}}{c} \right) \\
	& \hspace{6mm} = \frac{1}{\varphi(Q/q)} \sum_{\chi \mod(Q/q)} \chi(a) \sum_{\substack{(c,Q/q)=1}} \overline{\chi(c)} \frac{S(m,n;cq)}{cq} \phi\left(\frac{4 \pi \sqrt{|mn|}}{cq}\right)  \\ 
	& \hspace{6mm} = \frac{\sqrt{w_q}}{\varphi(Q/q)} \sum_{\kappa \in \{0,1\}} (-i)^{\kappa} \sum_{\substack{\chi \mod(Q/q)\\\chi(-1)=(-1)^{\kappa}}} \chi(a)  \\
	& \hspace{1.4cm}  \times \sum_{\substack{r \mod((q,Q/q))\\ (r,Q)=1}} \sum_{\substack{cq\sqrt{w_q} \in \mathcal{C}_{\infty,r/q}}} \frac{S_{\infty,r/q}^{\chi,\kappa}(m,nw_q;cq\sqrt{w_q})}{cq\sqrt{w_q}} \phi\left(\frac{4 \pi \sqrt{|mnw_q|}}{cq\sqrt{w_q}}\right)\\
	& \hspace{6mm} = \frac{\sqrt{w_q}}{\varphi(Q/q)} \sum_{\kappa \in \{0,1\}} (-i)^{\kappa} \sum_{\substack{\chi \mod(Q/q)\\\chi(-1)=(-1)^{\kappa}}} \chi(a) \\
	& \hspace{1.4cm} \times \sum_{\substack{r \mod((q,Q/q))\\ (r,Q)=1}} \left( \epsilon \mathcal{H}_{\infty,r/q}^{\chi,\kappa}(m,nw_q;\phi)+\mathcal{M}_{\infty,r/q}^{\chi,\kappa,\epsilon,\pm}(m,nw_q;\phi)+\mathcal{E}_{\infty,r/q}^{\chi,\kappa,\epsilon,\pm}(m,nw_q;\phi)  \right).
	\label{eq:Kloos-dyadic-smooth}
\end{aligned}
\end{equation}
We shall deal with these terms in the remaining sections.

\subsection{Transform estimates}

The next Lemma provides us with bounds on the integral transforms of $\phi$.

\begin{lem} With $\phi$ given as above and $\kappa \in \{0,1\}$, we have for any $\delta>0$
	\begin{align}
		|\widetilde{\phi}(t)| & \ll  t^{-1} \log(2t)^{2/3}, \quad \forall t \ge 1, \label{eq:trans-est-hol-triv} \\ |\widetilde{\phi}(t)| & \ll t^{-1/2} e^{-\frac{2}{5}t}, \quad \forall t \ge 5X, \label{eq:trans-est-hol-tail} \\
		|\widehat{\phi}^+_{\kappa}(t)| &\ll_{\delta} (1+|t|)^{\kappa} \frac{1+|\log(X)|+X^{-2|\Im t|}}{1+X} , \quad \forall t \in \mathbb{R} \cup ]-(\tfrac{1}{2}-\delta)i,(\tfrac{1}{2}-\delta)i[, \label{eq:trans-est-maass-plus-triv}\\ 
		|\widehat{\phi}^+_{\kappa}(t)| &\ll  |t|^{\kappa} \left( |t|^{-3/2}+|t|^{-2}X \right) \min\left\{ 1, \frac{C}{B|t|} \right\}, \quad \forall |t|\ge 1, \label{eq:trans-est-maass-plus-tail}\\
		|\widehat{\phi}^-_{0}(t)| &\ll_{\delta}  \frac{1+|\log(X)|+X^{-2|\Im t|}}{1+X} , \quad \forall t \in \mathbb{R} \cup ]-(\tfrac{1}{2}-\delta)i,(\tfrac{1}{2}-\delta)i[, \label{eq:trans-est-maass-minus-triv} \\
		|\widehat{\phi}^-_{0}(t)| &\ll \left( |t|^{-3/2}+|t|^{-2}X \right) \min\left\{ 1, \frac{C}{B|t|} \right\}, \quad \forall |t| \ge 1. \label{eq:trans-est-maass-minus-tail}
	\end{align}
\end{lem} 
\begin{proof}
	The estimates \eqref{eq:trans-est-hol-triv} \eqref{eq:trans-est-hol-triv} may be found in \cite[Lem 7]{twistlinnik}, whereas \eqref{eq:trans-est-maass-plus-triv}-\eqref{eq:trans-est-maass-minus-tail} for $\kappa=0$ may be found in \cite[Lem 7.1]{IwDes}. The proof for $\kappa=1$ is almost identical to the one for $\kappa=0$. 
Hence, we refer the reader to \cite[Lem 7.1]{IwDes} for this case as well.
\end{proof}

\subsection{Holomorphic spectrum}

In this section, we shall prove for $m,n\in \mathbb{N}$ and $\kappa\in \{0,1\}$ that
\begin{multline}
	 \frac{\sqrt{w_q}}{\varphi(Q/q)}\sum_{\substack{\chi \mod(Q/q) \\ \chi(-1)=(-1)^{\kappa}}} \sum_{\substack{r \mod((q,Q/q)) \\(r,Q)=1 }} \mathcal{H}^{\chi,\kappa}_{\infty,r/q}(m,nw_q;\phi) \\
	\prec \frac{Q^{1/2}(q,Q/q)^{1/2}}{q^{1/2}}  \left(   1+X +  \frac{m^{1/3}(q,m)^{1/4}+n^{1/3}(q,n)^{1/4}}{(q^{2},Q)^{1/12}Q^{1/2}}+ \frac{(mn)^{1/3}(q,m)^{1/4}(q,n)^{1/4}}{(q^2,Q)^{1/6}Q}  \right).
	\label{eq:hol-contribution}
\end{multline}
To prove this, we split the sum inside $\mathcal{H}$ into dyadic sub-sums with $k \sim K$, where $K \ge 1$. An application of Cauchy--Schwarz shows

\begin{equation*} \begin{aligned}
		&\sum_{\substack{\chi \mod(Q/q) \\ \chi(-1)=(-1)^{\kappa}}} \sum_{\substack{r \mod((q,Q/q)) \\(r,Q)=1 }}  \sum_{\substack{k \equiv \kappa \,(2)\\ 2 \le k \sim K}} \sum_{f \in \mathcal{B}_{\rm H}^k(\Gamma_0(Q),\chi)} \frac{\Gamma(k)}{(4\pi \sqrt{mnw_q})^{k-1}} |\overline{\psi_f(\infty,m)}\psi_f(r/q,nw_q)| |\widetilde{\phi}(k-1)| \\
		\ll & \left(\sum_{\substack{\chi \mod(Q/q) \\ \chi(-1)=(-1)^{\kappa}}} \sum_{\substack{r \mod((q,Q/q)) \\(r,Q)=1 }}  \sum_{\substack{k \equiv \kappa \,(2)\\ 2 \le k \sim K}} \sum_{f \in \mathcal{B}_{\rm H}^k(\Gamma_0(Q),\chi)} \frac{\Gamma(k)}{(4\pi nw_q)^{k-1}} |\psi_f(r/q,nw_q)|^2\right)^{1/2} \\ 
		& \times \left(    \sum_{\substack{\chi \mod(Q/q) \\ \chi(-1)=(-1)^{\kappa}}} (q,Q/q)  \sum_{\substack{k \equiv \kappa \,(2)\\ 2 \le k \sim K}} \sum_{f \in \mathcal{B}_{\rm H}^k(\Gamma_0(Q),\chi)} \frac{\Gamma(k)}{(4\pi m)^{k-1}} |\psi_f(\infty,m)|^2    \right)^{1/2} \sup_{2 \le k \sim K} |\widetilde{\phi}(k-1)|.
\end{aligned}\end{equation*}
The latter we bound using \eqref{eq:trans-est-hol-triv} and Proposition \ref{prop:Fourier-holo-diag-upper-bound} by
$$
\prec \frac{Q}{q} (q,Q/q) \left( K + \frac{m^{1/3}(q,m)^{1/4}+n^{1/3}(q,n)^{1/4}}{(q^{2},Q)^{1/12}Q^{1/2}}+ \frac{(mn)^{1/3}(q,m)^{1/4}(q,n)^{1/4}}{(q^2,Q)^{1/6}Q} \right).
$$
For $K \ge 1+5X$, we may use \eqref{eq:trans-est-hol-tail} instead to give the bound
$$
\prec \frac{Q}{q} (q,Q/q) e^{-K/5} \left( K^{3/2} + \frac{m^{1/3}(q,m)^{1/4}+n^{1/3}(q,n)^{1/4}}{(q^{2},Q)^{1/12}Q^{1/2}}K^{1/2}+ \frac{(mn)^{1/3}(q,m)^{1/4}(q,n)^{1/4}}{(q^2,Q)^{1/6}Q} \right).
$$
Summing over all dyadic intervals yields \eqref{eq:hol-contribution}.

\subsection{Maa{\ss} and Eisenstein spectrum}

In this section, we shall prove for $m\in \mathbb{N}$, $n \in \mathbb{Z}\backslash\{0\}$, and $\kappa\in \{0,1\}$ that
\begin{multline}
	\frac{\sqrt{w_q}}{\varphi(Q/q)}\sum_{\substack{\chi \mod(Q/q) \\ \chi(-1)=(-1)^{\kappa}}} \sum_{\substack{r \mod((q,Q/q)) \\(r,Q)=1 }} \left( {}^{\ast}\mathcal{M}^{\chi,\kappa,\epsilon}_{\infty,r/q}(m,nw_q;\phi) +\mathcal{E}^{\chi,\kappa,\epsilon}_{\infty,r/q}(m,nw_q;\phi) \right) \\
	\prec \frac{Q^{1/2}(q,Q/q)^{1/2}}{q^{1/2}} \left( \frac{C^{1/2}}{B^{1/2}}+X+ \frac{|m|^{1/4}(q,m)^{1/4}+|n|^{1/4}(q,n)^{1/4}}{Q^{1/2}} + \frac{|mn|^{1/4}(q,m)^{1/4}(q,n)^{1/4}}{Q}   \right),
	\label{eq:maass-contribution}
\end{multline}
where the asterisk indicates that the sum in $\mathcal{M}$ is only over the Maa{\ss} forms $f$ with $t_f \in \mathbb{R}$. As in the proof for holomorphic forms, we split the Maa{\ss} forms and Eisenstein integral into dyadic intervals $|t| \sim T$. By abuse of notation, we include the whole sum with $-1 \le t \le 1$ into $|t| \sim T$ when $T=1$. By Cauchy--Schwarz, we have

\begin{equation*} \begin{aligned}
		&\sum_{\substack{\chi \mod(Q/q) \\ \chi(-1)=(-1)^{\kappa}}} \sum_{\substack{r \mod((q,Q/q)) \\(r,Q)=1 }} \Biggl( \sum_{\substack{f \in \mathcal{B}^{\kappa}_{\rm M}(\Gamma_0(Q),\chi) \\ |t_f| \sim T}}  \frac{\sqrt{|mnw_q|}}{\cosh(\pi t_f)} |\overline{\rho_f(\infty,m)}\rho_f(r/q,nw_q)| |\widehat{\phi}^{\pm}_{\epsilon \kappa}(t_f)|  \\
		& \quad + \sum_{{\rm E} \in \mathcal{B}^{\kappa}_{\rm E}(\Gamma_0(Q),\chi)} \int_{|t|\sim T}  \frac{\sqrt{|mnw_q|}}{\cosh(\pi t)} |\overline{\rho_{\rm E}(\infty,m;t)}\rho_{\rm E}(r/q,nw_q;t)| |\widehat{\phi}^{\pm}_{\epsilon \kappa}(t)|  dt \Biggr)\\
		\ll & \Biggl(\sum_{\substack{\chi \mod(Q/q) \\ \chi(-1)=(-1)^{\kappa}}} \sum_{\substack{r \mod((q,Q/q)) \\(r,Q)=1 }} \Biggl( \sum_{\substack{f \in \mathcal{B}^{\kappa}_{\rm M}(\Gamma_0(Q),\chi) \\ |t_f| \sim T}}  \frac{|nw_q|}{\cosh(\pi t_f)} (1+|t_f|)^{\sign(n)\kappa} |\rho_f(r/q,nw_q)|^2  \\
		& \quad + \sum_{{\rm E} \in \mathcal{B}^{\kappa}_{\rm E}(\Gamma_0(Q),\chi)} \int_{|t|\sim T}  \frac{|nw_q|}{\cosh(\pi t)} (1+|t|)^{\sign(n)\kappa} |\rho_{\rm E}(r/q,nw_q;t)|^2  dt \Biggr)\Biggr)^{1/2} \\ 
		& \times \Biggl(   \sum_{\substack{\chi \mod(Q/q) \\ \chi(-1)=(-1)^{\kappa}}} (q,Q/q) \Biggl( \sum_{\substack{f \in \mathcal{B}^{\kappa}_{\rm M}(\Gamma_0(Q),\chi) \\ |t_f| \sim T}}  \frac{|m|}{\cosh(\pi t_f)} (1+|t_f|)^{\kappa} |\rho_f(\infty,m)|^2  \\
		&  \quad + \sum_{{\rm E} \in \mathcal{B}^{\kappa}_{\rm E}(\Gamma_0(Q),\chi)} \int_{|t|\sim T}  \frac{|m|}{\cosh(\pi t)}  (1+|t|)^{\kappa} |\rho_{\rm E}(r/q,nw_q;t)|^2  dt \Biggr)   \Biggr)^{1/2} \sup_{|t| \sim T} \frac{|\widehat{\phi}^{\pm}_{\epsilon \kappa}(t)|}{(1+|t|)^{\epsilon \kappa}}.
\end{aligned}\end{equation*}
We bound further using \eqref{eq:trans-est-maass-plus-triv}-\eqref{eq:trans-est-maass-minus-tail}, and Proposition \ref{prop:chi-avg-fourier-bound}:
\begin{multline}
\prec (q,Q/q)\frac{Q}{q} \left( T^2 + \frac{|m|^{1/4}(q,m)^{1/4}+|n|^{1/4}(q,n)^{1/4}}{Q^{1/2}} T + \frac{|mn|^{1/4}(q,m)^{1/4}(q,n)^{1/4}}{Q}  \right) \\ \times \min\left\{(1+X)^{-1} ,\left(T^{-3/2}+T^{-2}X  \right), \left(T^{-3/2}+T^{-2}X  \right)  \frac{C}{BT}\right\}.
\label{eq:maass-pre-contribution}
\end{multline}
We now sum up all of dyadic intervals where we make use of the first term in the minimum if $1\le T \le 1+X$, the second term if $1+X \le T \le \frac{C}{B}$, and the third term if $1+X, \frac{C}{B} \le T$. This yields \eqref{eq:maass-contribution}.

\subsection{Exceptional Maa{\ss} spectrum}

The exceptional spectrum, i.e.\@ those eigenfunctions $f \in \mathcal{B}_{\rm M}^{\kappa}(\Gamma_0(Q),\chi)$ with $\Im t_f \neq 0$ only arise when $\kappa=0$. Furthermore, we have $|\Im t_f| \le \vartheta$, where we may take $\vartheta=7/64$ by the work of Kim--Sarnak \cite{KimSarnak}. In this section, we shall prove that their contribution is bounded as follows:
\begin{equation}\begin{aligned}
	&\frac{\sqrt{w_q}}{\varphi(Q/q)}\sum_{\substack{\chi \mod(Q/q) \\ \chi(-1)=1}} \sum_{\substack{r \mod((q,Q/q)) \\(r,Q)=1 }}  {}^{\dagger}\mathcal{M}^{\chi,0,\epsilon}_{\infty,r/q}(m,nw_q;\phi)\\
	\prec 
	&  \frac{Q^{1/2}(q,Q/q)^{1/2}}{q^{1/2}} \Biggl( 1 + \frac{|m|^{1/4}(q,m)^{1/4}+|n|^{1/4}(q,n)^{1/4}}{Q^{1/2}} + \frac{|mn|^{1/4}(q,m)^{1/4}(q,n)^{1/4}}{Q}  \Biggr) \\
	& \quad \times \Biggl( 1+\frac{|m|(q,m)+|n|(q,n)}{Q^{2}}X^{-1} + \frac{|mn|(q,m)(q,n)}{Q^4}  X^{-2} \Biggr)^{\vartheta},
	\label{eq:maass-exc-contribution}
	\end{aligned}
\end{equation}
where the dagger stands for restricting the sum to the exceptional eigenfunctions. By Cauchy--Schwarz, we have
\begin{equation*} \begin{aligned}
		&\sum_{\substack{\chi \mod(Q/q) \\ \chi(-1)=1}} \sum_{\substack{r \mod((q,Q/q)) \\(r,Q)=1 }} \sum_{\substack{f \in \mathcal{B}^{0}_{\rm M}(\Gamma_0(Q),\chi) \\ 0 <|\Im t_f| \le \vartheta}}  \frac{\sqrt{|mnw_q|}}{\cosh(\pi t_f)} |\overline{\rho_f(\infty,m)}\rho_f(r/q,nw_q)| |\widehat{\phi}^{\pm}_{\epsilon \kappa}(t_f)|  \\
		\ll & \left( \sum_{\substack{\chi \mod(Q/q) \\ \chi(-1)=1}} \sum_{\substack{r \mod((q,Q/q)) \\(r,Q)=1 }}  \sum_{\substack{f \in \mathcal{B}^{0}_{\rm M}(\Gamma_0(Q),\chi) \\ 0 <|\Im t_f| \le \vartheta}}  \frac{|nw_q|}{\cosh(\pi t_f)}  |\rho_f(r/q,nw_q)|^2 |\widehat{\phi}^{\pm}_{\epsilon \kappa}(t_f)| \right)^{1/2}  \\
		& \times \left(   \sum_{\substack{\chi \mod(Q/q) \\ \chi(-1)=1}} (q,Q/q) \sum_{\substack{f \in \mathcal{B}^{0}_{\rm M}(\Gamma_0(Q),\chi) \\ 0 <|\Im t_f| \le \vartheta}}  \frac{|m|}{\cosh(\pi t_f)}  |\rho_f(\infty,m)|^2 |\widehat{\phi}^{\pm}_{\epsilon \kappa}(t_f)| \right)^{1/2}.
\end{aligned}\end{equation*}
We bound $|\widehat{\phi}^{\pm}_{\epsilon\kappa}|$ using \eqref{eq:trans-est-maass-plus-triv}, \eqref{eq:trans-est-maass-minus-triv}. If $X \ge 1$, then we bound the spectral sum using Proposition \ref{prop:chi-avg-fourier-bound} and get the same bound as in \eqref{eq:maass-pre-contribution} with $T \sim 1$. If $X \le 1$, we may argue as in \cite[\S 2.1]{FouvryMichel} and use of Proposition \ref{prop:exceptional-chi-avg-fourier-bound} instead to find
\begin{equation*}\begin{aligned}
		& \sum_{\substack{\chi \mod(Q/q) \\ \chi(-1)=1}}  \sum_{\substack{f \in \mathcal{B}^{0}_{\rm M}(\Gamma_0(Q),\chi) \\ 0 <|\Im t_f| \le \vartheta}}  \frac{|m|}{\cosh(\pi t_f)}  |\rho_f(\infty,m)|^2 X^{-2|t_f|} \\
 		\ll & \sum_{\substack{\chi \mod(Q/q) \\ \chi(-1)=1}} \sum_{\substack{f \in \mathcal{B}^{0}_{\rm M}(\Gamma_0(Q),\chi) \\ 0 <|\Im t_f| \le \vartheta}}  \frac{|m|}{\cosh(\pi t_f)}  |\rho_f(\infty,m)|^2  \\  & \quad \times  \min\left\{X^{-1}, 1+\frac{Q^2}{|m|(q,m)}\right\}^{2|t_f|} \left(  \min\left\{X^{-1}, 1+\frac{Q^2}{|m|(q,m)}\right\}  X \right)^{-2\vartheta} \\
 		\prec &  \frac{Q}{q} \left(1+\frac{|m|^{1/2}(q,m)^{1/2}}{Q} \right) \left( 1+  \left(\frac{Q}{|m|^{1/2}(q,m)^{1/2}}\right)^{-4\vartheta} X^{-2\vartheta} \right).
\end{aligned}\end{equation*}
Likewise, we find
\begin{multline*}
		 \sum_{\substack{\chi \mod(Q/q) \\ \chi(-1)=1}} \sum_{\substack{r \mod((q,Q/q)) \\(r,Q)=1 }}  \sum_{\substack{f \in \mathcal{B}^{0}_{\rm M}(\Gamma_0(Q),\chi) \\ 0 <|\Im t_f| \le \vartheta}}  \frac{|nw_q|}{\cosh(\pi t_f)}  |\rho_f(r/q,nw_q)|^2 |\widehat{\phi}^{\pm}_{\epsilon \kappa}(t_f)| \\
		\prec   (q,Q/q)\frac{Q}{q} \left(1+\frac{|n|^{1/2}(q,n)^{1/2}}{Q} \right) \left( 1+  \left(\frac{Q}{|n|^{1/2}(q,n)^{1/2}}\right)^{-4\vartheta} X^{-2\vartheta} \right).
 \end{multline*}
We thus conclude \eqref{eq:maass-exc-contribution}.

\subsection{Final bound}

By adding up the contributions \eqref{eq:Klooster-AP-trivial-short} \eqref{eq:hol-contribution}, \eqref{eq:maass-contribution}, \eqref{eq:maass-exc-contribution} and choosing
$$
B \asymp \frac{Q C^{2/3} (q,Q/q)^{1/3}}{q^{1/3}(m,n,q)^{1/3}},
$$
we find that
\begin{multline}
	\sum_{\substack{C  <  c \le 2C \\ c \equiv aq \mod(Q)}} \frac{S(m,n;c)}{c} \prec \frac{(m,n)}{C^{1/2}}  + \frac{(m,n,q)^{1/6}(q,Q/q)^{1/3}}{q^{1/3}}C^{1/6} \\
	+\frac{Q^{1/2}(q,Q/q)^{1/2}}{q^{1/2}} \Biggl( 1+\frac{|mn|^{1/2}}{C}+\frac{|m|^{1/3}(q,m)^{1/4}+|n|^{1/3}(q,n)^{1/4}}{(q^2,Q)^{1/12}Q^{1/2}} + \frac{|mn|^{1/3}(q,m)^{1/4}(q,n)^{1/4}}{(q^2,Q)^{1/6}Q} \Biggr)\\
	+ \frac{Q^{1/2}(q,Q/q)^{1/2}}{q^{1/2}} \Biggl( 1+\frac{|m|^{1/4}(q,m)^{1/4}+|n|^{1/4}(q,n)^{1/4}}{Q^{1/2}} + \frac{|mn|^{1/4}(q,m)^{1/4}(q,n)^{1/4}}{Q} \Biggr) \\
	\times \Biggl( 1+\frac{|m|(q,m)+|n|(q,n)}{|mn|^{1/2}} \frac{C}{Q^2} + (q,m)(q,n) \frac{C^2}{Q^4} \Biggr)^{\vartheta}
	\label{eq:dyadic-bound}
\end{multline}
valid for 
$$
C \ge \frac{Q^3(q,Q/q)}{q(m,n,q)}.
$$
For smaller $C$, we may use the bound \eqref{eq:Klooster-AP-trivial}, in order to prove the following theorem.
\begin{thm} Let $m,n \in \mathbb{Z}\backslash\{0\}$ be two non-zero integers. Further let $a,q,Q \in \mathbb{N}$ be integers with $q \mid Q$ and $(a,Q/q)=1$. Then, we have
	\begin{multline*}
		\sum_{\substack{c \le C \\ c \equiv aq \mod(Q)}} \frac{S(m,n;c)}{c} \prec \frac{(m,n,q)^{1/2}}{q^{1/2}} + \frac{(m,n)}{Q^{1/2}}  + \frac{(m,n,q)^{1/6}(q,Q/q)^{1/3}}{q^{1/3}}C^{1/6} \\
		+\frac{Q^{1/2}(q,Q/q)^{1/2}}{q^{1/2}} \Biggl( 1+\frac{|mn|^{1/2}}{Q}+\frac{|m|^{1/3}(q,m)^{1/4}+|n|^{1/3}(q,n)^{1/4}}{(q^2,Q)^{1/12}Q^{1/2}} + \frac{|mn|^{1/3}(q,m)^{1/4}(q,n)^{1/4}}{(q^2,Q)^{1/6}Q} \Biggr)\\
		+ \frac{Q^{1/2}(q,Q/q)^{1/2}}{q^{1/2}} \Biggl( 1+\frac{|m|^{1/4}(q,m)^{1/4}+|n|^{1/4}(q,n)^{1/4}}{Q^{1/2}} + \frac{|mn|^{1/4}(q,m)^{1/4}(q,n)^{1/4}}{Q} \Biggr) \\
		\times \Biggl( 1+\frac{|m|(q,m)+|n|(q,n)}{|mn|^{1/2}} \frac{C}{Q^2} + (q,m)(q,n) \frac{C^2}{Q^4} \Biggr)^{\vartheta}.
	\end{multline*}
	\label{thm:mnqQ-depence}
\end{thm}

\begin{rem} The $m,n$ depence can likely be improved by considering a Hecke eigenbasis and appealing to (the progress towards) the Ramanujan--Petersson conjecture. This would require computing the Fourier coefficients at the cusps considered in this paper in terms of Hecke eigenvalues. We refer to \cite{Min-Fourier} for how one may proceed in computing these and to \cite{BlomerKloos}, \cite{Humphries16} for a Hecke eigenbasis.
	\end{rem}

\bibliography{RafBib}

\begin{thebibliography}{EASS22}

\bibitem[BB12]{BaiBrow}
S.~Baier and T.~D. Browning.
\newblock Inhomogeneous quadratic congruences.
\newblock {\em Funct. Approx. Comment. Math.}, 47(part 2):267--286, 2012.

\bibitem[BHM07]{MR2382859}
V.~Blomer, G.~Harcos, and Ph. Michel.
\newblock Bounds for modular {$L$}-functions in the level aspect.
\newblock {\em Ann. Sci. \'{E}cole Norm. Sup. (4)}, 40(5):697--740, 2007.

\bibitem[BKS19]{S3covexp}
T.~D. Browning, V.~V. Kumaraswamy, and R.~S. Steiner.
\newblock Twisted {L}innik implies optimal covering exponent for {$S^3$}.
\newblock {\em Int. Math. Res. Not. IMRN}, (1):140--164, 2019.

\bibitem[BM15]{BlomerKloos}
V.~Blomer and D.~Mili{\'c}evi{\'c}.
\newblock Kloosterman sums in residue classes.
\newblock {\em J. Eur. Math. Soc. (JEMS)}, 17(1):51--69, 2015.

\bibitem[DFI02]{DFI2}
W.~Duke, J.~B. Friedlander, and H.~Iwaniec.
\newblock The subconvexity problem for {A}rtin {$L$}-functions.
\newblock {\em Invent. Math.}, 149(3):489--577, 2002.

\bibitem[DI83]{IwDes}
J.-M. Deshouillers and H.~Iwaniec.
\newblock Kloosterman sums and {F}ourier coefficients of cusp forms.
\newblock {\em Invent. Math.}, 70(2):219--288, 1982/83.

\bibitem[DM18]{DrapAutomaton}
S.~Drappeau and C.~M\"{u}llner.
\newblock Exponential sums with automatic sequences.
\newblock {\em Acta Arith.}, 185(1):81--99, 2018.

\bibitem[Dra17]{DrapKloost}
S.~Drappeau.
\newblock Sums of {K}loosterman sums in arithmetic progressions, and the error
  term in the dispersion method.
\newblock {\em Proc. Lond. Math. Soc. (3)}, 114(4):684--732, 2017.

\bibitem[EASS22]{Klooster-low-complexity-seq}
E.~H. El~Abdalaoui, I.~E. Shparlinski, and R.~S. Steiner.
\newblock {C}howla and {S}arnak conjectures for {K}loosterman sums.
\newblock {\em {P}reprint}, 2022.
\newblock {\tt arXiv:2211.00379}.

\bibitem[FM07]{FouvryMichel}
\'{E}. Fouvry and Ph. Michel.
\newblock Sur le changement de signe des sommes de {K}loosterman.
\newblock {\em Ann. of Math. (2)}, 165(3):675--715, 2007.

\bibitem[GHL15]{Min-Fourier}
D.~Goldfeld, J.~Hundley, and M.~Lee.
\newblock Fourier expansions of {$GL(2)$} newforms at various cusps.
\newblock {\em Ramanujan J.}, 36(1-2):3--42, 2015.

\bibitem[GS83]{GoldfeldSarnak}
D.~Goldfeld and P.~Sarnak.
\newblock Sums of {K}loosterman sums.
\newblock {\em Invent. Math.}, 71(2):243--250, 1983.

\bibitem[GS12]{APKloosterman}
S.~Ganguly and J.~Sengupta.
\newblock Sums of {K}loosterman sums over arithmetic progressions.
\newblock {\em Int. Math. Res. Not. IMRN}, (1):137--165, 2012.

\bibitem[GS21]{APKloostermanCorr}
S.~Ganguly and J.~Sengupta.
\newblock Corrigendum to ``{S}ums of {K}loosterman sums over arithmetic
  progressions''.
\newblock {\em Int. Math. Res. Not. IMRN}, (18):14405--14408, 2021.

\bibitem[Hum18]{Humphries16}
P.~Humphries.
\newblock Density theorems for exceptional eigenvalues for congruence
  subgroups.
\newblock {\em Algebra Number Theory}, 12(7):1581--1610, 2018.

\bibitem[IK04]{ANT}
H.~Iwaniec and E.~Kowalski.
\newblock {\em Analytic number theory}, volume~53 of {\em American Mathematical
  Society Colloquium Publications}.
\newblock American Mathematical Society, Providence, RI, 2004.

\bibitem[ILS00]{lowlying}
H.~Iwaniec, W.~Luo, and P.~Sarnak.
\newblock Low lying zeros of families of {$L$}-functions.
\newblock {\em Inst. Hautes \'{E}tudes Sci. Publ. Math.}, (91):55--131 (2001),
  2000.

\bibitem[Iwa84]{IwPrimeGeodesic}
H.~Iwaniec.
\newblock Prime geodesic theorem.
\newblock {\em J. Reine Angew. Math.}, 349:136--159, 1984.

\bibitem[Iwa97]{IwClassic}
H.~Iwaniec.
\newblock {\em Topics in classical automorphic forms}, volume~17 of {\em
  Graduate Studies in Mathematics}.
\newblock American Mathematical Society, Providence, RI, 1997.

\bibitem[Iwa02]{IwSpecMeth}
H.~Iwaniec.
\newblock {\em Spectral methods of automorphic forms}, volume~53 of {\em
  Graduate Studies in Mathematics}.
\newblock American Mathematical Society, Providence, RI; Revista Matem\'atica
  Iberoamericana, Madrid, second edition, 2002.

\bibitem[Kim03]{KimSarnak}
H.~H. Kim.
\newblock Functoriality for the exterior square of {${\rm GL}_4$} and the
  symmetric fourth of {${\rm GL}_2$}.
\newblock {\em J. Amer. Math. Soc.}, 16(1):139--183, 2003.
\newblock With appendix 1 by Dinakar Ramakrishnan and appendix 2 by Henry H.
  Kim and Peter Sarnak.

\bibitem[KL13]{Knightly-Li}
A.~Knightly and C.~Li.
\newblock Kuznetsov's trace formula and the {H}ecke eigenvalues of {M}aass
  forms.
\newblock {\em Mem. Amer. Math. Soc.}, 224(1055):vi+132, 2013.

\bibitem[Kow19]{Kowalski-notes-Klooster-low-complexity}
E.~Kowalski.
\newblock Unmotivated ergodic averages.
\newblock {\em {N}otes}, 2019.
\newblock Available at
  \url{https://people.math.ethz.ch/~kowalski/ergodic-trace.pdf}.

\bibitem[Kuz80]{Kuz1}
N.~V. Kuznetsov.
\newblock The {P}etersson conjecture for cusp forms of weight zero and the
  {L}innik conjecture. {S}ums of {K}loosterman sums.
\newblock {\em Mat. Sb. (N.S.)}, 111(153)(3):334--383, 479, 1980.

\bibitem[KY19]{Young-Kiral-Kloosterman}
E.~M. K{\i}ral and M.~P. Young.
\newblock Kloosterman sums and {F}ourier coefficients of {E}isenstein series.
\newblock {\em Ramanujan J.}, 49(2):391--409, 2019.

\bibitem[Lin63]{LinnikKloosterman}
Ju.~V. Linnik.
\newblock Additive problems and eigenvalues of the modular operators.
\newblock In {\em Proc. {I}nternat. {C}ongr. {M}athematicians ({S}tockholm,
  1962)}, pages 270--284. Inst. Mittag-Leffler, Djursholm, 1963.

\bibitem[Mic95]{MichelAutourKloos}
P.~Michel.
\newblock Autour de la conjecture de {S}ato-{T}ate pour les sommes de
  {K}loosterman. {I}.
\newblock {\em Invent. Math.}, 121(1):61--78, 1995.

\bibitem[Pro05]{Proskurin}
N.~V. Proskurin.
\newblock On the general {K}loosterman sums.
\newblock {\em Journal of Mathematical Sciences}, 129(3):3874--3889, Sep 2005.

\bibitem[Ran77]{MFaF}
R.~A. Rankin.
\newblock {\em Modular forms and functions}.
\newblock Cambridge University Press, Cambridge-New York-Melbourne, 1977.

\bibitem[Sel65]{SelbergFourier}
A.~Selberg.
\newblock On the estimation of {F}ourier coefficients of modular forms.
\newblock In {\em Proc. {S}ympos. {P}ure {M}ath., {V}ol. {VIII}}, pages 1--15.
  Amer. Math. Soc., Providence, R.I., 1965.

\bibitem[ST09]{SarTsim}
P.~Sarnak and J.~Tsimerman.
\newblock On {L}innik and {S}elberg's conjecture about sums of {K}loosterman
  sums.
\newblock In {\em Algebra, arithmetic, and geometry: in honor of {Y}u. {I}.
  {M}anin. {V}ol. {II}}, volume 270 of {\em Progr. Math.}, pages 619--635.
  Birkh\"auser Boston, Inc., Boston, MA, 2009.

\bibitem[Ste18]{thesis}
R.~S. Steiner.
\newblock The harmonic conjunction of automorphic forms and the
  {H}ardy--{L}ittlewood circle method.
\newblock {\em Ph.D. thesis, University of Bristol}, 2018.

\bibitem[Ste19]{twistlinnik}
R.~S. Steiner.
\newblock On a twisted version of {L}innik and {S}elberg's conjecture on sums
  of {K}loosterman sums.
\newblock {\em Mathematika}, 65(3):437--474, 2019.

\bibitem[Wat44]{ToBF}
G.~N. Watson.
\newblock {\em A {T}reatise on the {T}heory of {B}essel {F}unctions}.
\newblock Cambridge University Press, Cambridge, England; The Macmillan
  Company, New York, 1944.

\end{thebibliography}
\end{document}